\documentclass{siamltex}

\usepackage[nosumlimits]{amsmath}
\usepackage{amsfonts,amssymb}
\usepackage{graphicx}
\graphicspath{{./Figures/}}
\usepackage{color}
\usepackage{ulem}
\usepackage{multirow}
\usepackage{multicol}
\usepackage{algorithm}
\usepackage{algpseudocode}
\usepackage{algorithmicx}

\usepackage{cite} 
\usepackage[dvipsnames]{xcolor}

\DeclareMathAlphabet{\mathbf}{OT1}{cmr}{bx}{it}

\newcommand{\vb}{{\mathbf b}}
\newcommand{\vc}{{\mathbf c}}

\newcommand{\ve}{{\mathbf e}}

\newcommand{\vu}{{\mathbf u}}
\newcommand{\vv}{{\mathbf v}}

\newcommand{\vw}{{\mathbf w}}
\newcommand{\vx}{{\mathbf x}}
\newcommand{\vy}{{\mathbf y}}

\newcommand{\vnull}{\boldsymbol{0}}
\newcommand{\spK}{{\cal K}}

\renewcommand{\d}{\,\mathrm{d}}
\newcommand{\R}{\mathbb{R}}
\newcommand{\N}{\mathbb{N}}

\newcommand{\C}{\mathbb{C}}
\newcommand{\Cn}{\mathbb{C}^n}
\newcommand{\Cnn}{\mathbb{C}^{n \times n}}

\newcommand{\Cmm}{\mathbb{C}^{m \times m}}

\newcommand{\dmu}{\d\mu(t)}

\newcommand{\calU}{\mathcal{U}}
\newcommand{\calV}{\mathcal{V}}

\newcommand{\calW}{W}
\newcommand{\calO}{\mathcal{O}}
\newcommand{\calA}{\mathcal{A}}
\newcommand{\calT}{\mathcal{T}}

\newcommand{\bbD}{\mathbb{D}}
\newcommand{\bbE}{\mathbb{E}}
\newcommand{\barC}{\overline{\mathbb{C}}}

\DeclareMathOperator{\nnz}{nnz}
\DeclareMathOperator{\dist}{dist}

\DeclareMathOperator{\Span}{span}
\newcommand{\spec}{\Lambda}

\DeclareMathOperator{\trace}{trace}

\newtheorem{remark}[theorem]{Remark}
\let\oldremark\remark
\renewcommand{\remark}{\oldremark\normalfont}

\newcommand{\lmin}{\lambda_{\min}}
\newcommand{\lmax}{\lambda_{\max}}
\newtheorem{example}[theorem]{Example}
\let\oldexample\example
\renewcommand{\example}{\oldexample\normalfont}

\title{Low-rank updates of matrix functions}
\author{Bernhard Beckermann\thanks{Laboratoire Painlev\'e UMR 8524, UFR Math\'ematiques, Univ.\ Lille, 59655 Villeneuve d'Ascq, France. E-mail: {\tt bbecker@math.univ-lille1.fr}. Supported  in  part  by  the  Labex  CEMPI (ANR-11-LABX-0007-01).} \and Daniel Kressner\thanks{MATHICSE-ANCHP, \'Ecole Polytechnique F\'ed\'erale de Lausanne, Station 8, 1015 Lausanne, Switzerland. E-mail: \texttt{daniel.kressner@epfl.ch}} \and Marcel Schweitzer\thanks{MATHICSE-ANCHP, \'Ecole Polytechnique F\'ed\'erale de Lausanne, Station 8, 1015 Lausanne, Switzerland. E-mail: \texttt{marcel.schweitzer@epfl.ch}. The work of Marcel Schweitzer has been supported by the SNSF project \emph{Low-rank updates of matrix functions and fast eigenvalue solvers}.}}
\date{\today}
\begin{document}
\renewcommand{\thefootnote}{\fnsymbol{footnote}}
\maketitle \pagestyle{myheadings} \thispagestyle{plain}
\markboth{B. BECKERMANN, D. KRESSNER, AND M. SCHWEITZER}{LOW-RANK UPDATES OF MATRIX FUNCTIONS}

\begin{abstract}
We consider the task of updating a matrix function $f(A)$ when the matrix $A\in\Cnn$ is subject to a low-rank modification. In other words, we aim at approximating $f(A+D)-f(A)$ for a matrix $D$ of rank $k \ll n$. The approach proposed in this paper attains efficiency by projecting onto tensorized Krylov subspaces produced by matrix-vector multiplications with $A$ and $A^*$. We prove the approximations obtained from $m$ steps of the proposed methods are exact if $f$ is a polynomial of degree at most $m$ and use this as a basis for proving a variety of convergence results, in particular for the matrix exponential and for Markov functions. We illustrate the performance of our method by considering various examples from network analysis, where our approach can be used to cheaply update centrality and communicability measures. 
\end{abstract}



\section{Introduction}

This work is concerned with the problem of updating a matrix function $f(A)$ when $A\in \C^{n\times n}$ is
subject to a low-rank modification.
More specifically, given $D \in \C^{n\times n}$ of rank $r\ll n$ we aim at computing
the difference
\begin{equation} \label{eq:matrixfunupdate}
f(A + D) -  f(A),
\end{equation}
at a cost significantly smaller than the cost of computing $f(A + D)$ from scratch.
Throughout this work, we assume that $f:\Omega \to \C$ is analytic on some domain $\Omega$
containing $\Lambda(A)$ and $\Lambda(A+D)$, the spectra of $A$ and $A+D$.

The generality of~\eqref{eq:matrixfunupdate} includes a wide scope of applications.
We will particularly focus on
measuring network properties, such as centrality and communicability, via applying the matrix exponential to the adjacency matrix of an undirected graph~\cite{EstradaHigham2010}.
Removing/inserting individual edges or nodes in the graph corresponds to rank-$2$ updates of $A$. Being able to solve~\eqref{eq:matrixfunupdate} efficiently therefore allows to quickly update these measures.
In fact, this task is explicitly needed when measuring the betweenness of a node~\cite[Sec. 2.3]{EstradaHigham2010}.

For $f(z) = z^{-1}$, the well-known Sherman-Morrison-Woodbury formula
implies that the difference~\eqref{eq:matrixfunupdate} has rank $r$ and can be directly obtained from applying $f(A) = A^{-1}$
to the low-rank factors of $D$. In particular, when $D$ has rank $1$ and can be written as $D = \vb\vc^*$ with vectors $\vb$ and $\vc$, we have
\begin{equation} \label{eq:sherman}
 (A + \vb\vc^*)^{-1} - A^{-1} = -\frac{1}{1+\vc^* A^{-1} \vb} A^{-1} \vb\vc^* A^{-1}.
\end{equation}
Unless $A$ is a scalar multiple of the identity matrix~\cite[Theorem 1.35]{Higham2008}, there is no such simple relation
between $f(A)$ and $f(A+\vb\vc^*)$ for a general analytic function $f$. In the case of a rational function $f \equiv r$ of degree $\delta$, 
Bernstein and Van Loan~\cite{BernsteinVanLoan2000} show that $r(A + \vb\vc^*) - r(A)$ has rank at most $\delta$ and provide an explicit formula for the rank-$\delta$ correction. This construction is based on explicit Krylov bases and potentially prone to numerical instability for larger degrees.

The Cauchy integral representation \begin{equation} \label{eq:contourintro} f(A) = \frac{1}{2\pi \mathrm{i}} \int_\Gamma f(z)
 (zI - A)^{-1}  \text{d}z\end{equation} provides a link between matrix functions and (shifted) inverses. In~\cite{Stange2011}, a combination of numerical quadrature with~\eqref{eq:sherman} has been explored for approximating $\exp(A + \vb\vc^*) -  \exp(A)$. However, such an approach suffers from a number of drawbacks. Most importantly, the choice of a contour $\Gamma$ that is equally good for $A$ and $A+\vb\vc^*$ is highly nontrivial, in particular in the non-Hermitian case. Also, the evaluation of the approximation amounts to the solution of a differently shifted linear system for each quadrature point. Although not used in our algorithms, integral representations will play a role in their convergence analysis, see Section~\ref{sec:convergence_integral}.

The approach proposed in this work avoids the need for choosing a contour or an explicit rational approximation of $f$. Inspired by existing Krylov subspace methods for matrix equations~\cite{Simoncini2016} and linear systems with tensor product structure~\cite{Kressner2010}, we make use of tensorized subspaces. More specifically, given orthonormal bases $U_m,V_m$ for $m$-dimensional Krylov subspaces involving the matrices $A,A^*$ and the starting vectors $\vb,\vc$, we construct an approximation of the form
\begin{equation*}
 f(A+\vb\vc^*) - f(A) \approx U_m X_m(f) V_m^*,
\end{equation*}
with a well-chosen small matrix $X_m(f) \in \Cmm$. In turn, when $m\ll n$, the computational cost is dominated by $m$ matrix-vector products with $A$ and $A^*$. As we will demonstrate for a variety of examples, this dramatically reduces the computational effort compared to computing $f(A+\vb\vc^*)$ from scratch, even when only selected quantities -- such as the diagonal or trace of $f(A+\vb\vc^*)$ -- are required. 

The rest of this paper is organized as follows. In section~\ref{sec:algorithms}, we introduce Krylov subspace algorithms for approximating the update~\eqref{eq:matrixfunupdate} when $D$ is a matrix of rank one, both for the Hermitian and the non-Hermitian case, and discuss the extension to higher rank. In section~\ref{sec:exactness}, we prove that the approximation obtained from $m$ steps of our Krylov subspace algorithm is exact when $f$ is a polynomial of degree at most $m$. This result forms the basis of the convergence analysis presented in sections~\ref{sec:convergenceanalysis} and~\ref{sec:convergence_integral}. Precisely, the convergence of our Krylov approximations is linked to certain polynomial approximation problems in section~\ref{sec:convergenceanalysis}, while section~\ref{sec:convergence_integral} contains results that are based on exploiting an integral representation of $f(A)$, e.g., for general analytic functions and for Markov functions. Applications for our methods are described in Section~\ref{sec:experiments}, with special focus on up-/down-dating of communicability measures in network analysis, and the efficiency of our approach is illustrated by numerical experiments from this area. Concluding remarks are given in Section~\ref{sec:conclusion}.

We use $\|\cdot\|$ to denote the Euclidean norm of a vector or the induced spectral norm of a matrix.

\section{Krylov projection algorithms}\label{sec:algorithms}

In the following, we assume that $f(A)$ or at least the part relevant to the application (e.g., the diagonal of $f(A)$) have already been computed. For the moment, we continue to assume a rank-$1$ modification, that is, we consider the approximation of $f(A+\vb\vc^*)-f(A)$ for vectors $\vb,\vc$. In Remark~\ref{remark:higherrank} below, we will comment on the extension to higher ranks.

\subsection{Hermitian rank-1 updates} \label{sec:hermitian}

Because it is conceptually simpler, we first discuss the Hermitian case: $A = A^\ast$ and $\vb = \vc$.

The first step of our algorithm consists of constructing an orthonormal basis for a Krylov subspace of $A$ with starting vector $\vb$. Given $m\le n$, such a Krylov subspace takes the form 
\begin{equation*}
\calU_m := \spK_m(A,\vb) = \Span\{ \vb, A\vb, A^2\vb, \ldots, A^{m-1} \vb\}.
\end{equation*}
For simplicity, we suppose that $\calU_m$ has dimension $m$, which is generically satisfied. Applying $m$ steps of the Lanczos method~\cite{Lanczos1950} (possibly with reorthogonalization) results in the Lanczos relation
\begin{equation}
AU_m = U_mG_m + \beta_{m+1} \vu_{m+1}\ve_m^\ast, \label{eq:arnoldi_relationU}
\end{equation}
where the columns of $U_m \in \C^{n\times m}$ form an orthonormal basis of $\calU_m$, the matrix $G_m = U_m^\ast A U_m \in \C^{m\times m}$ is tridiagonal, and $\ve_m$ denotes the $m$th unit vector of length $m$. The first column of $U_m$ is a scalar multiple of $\vb$ and, without loss of generality, we may assume that $U_m^\ast \vb = \|\vb\| \ve_1$ with the first unit vector $\ve_1 \in \C^m$.

The second step of our algorithm then chooses an approximation in the tensorized subspace $\calU_m \otimes \calU_m$, that is,
\begin{equation}\label{eq:krylov_update_hermitian}
 f(A+\vb\vb^*)-f(A) \approx U_m X_m(f) U_m^*
\end{equation}
for some matrix $X_m(f) \in \C^{m \times m}$. In the spirit of Krylov subspace methods for matrix equations~\cite{Simoncini2016}, we choose $X_m(f)$ as the solution of the compressed problem:
\begin{equation}\label{eq:Xm_hermitian_difference}
 X_m(f) = f\big(U_m^* (A+\vb\vb^*) U_m\big) -  f\big( U_m^* A U_m \big) = f\big(G_m + \|\vb\|^2 \ve_1 \ve_1^\ast \big) - f( G_m ),
\end{equation}
where we assume that $f$ is defined on the spectra of $G_m$ and $G_m + \|\vb\|^2 \ve_1 \ve_1^\ast$.
Below, in Theorem~\ref{the:polynomial_exactness}, we will see that this choice of $X_m(f)$ leads to an approximation that is exact when $f$ is a polynomial of degree at most $m$.

The resulting method is summarized in Algorithm~\ref{alg:krylovhermitian}. The tridiagonal matrix $G_m$ from~\eqref{eq:arnoldi_relationU} is built from the orthogonalization coefficients $\alpha_j, \beta_j$ as
$$G_m = \begin{bmatrix}
\alpha_1 & \beta_2 & & & \\ 
\beta_2 & \alpha_2 & \beta_3 & & \\
 & \ddots & \ddots & \ddots & \\
 & & \beta_{m-1} & \alpha_{m-1} & \beta_m \\
&  & & \beta_m & \alpha_m \\
\end{bmatrix}.$$

A trivial modification of this algorithm can be used to approximate $f(A-\vb\vb^*)-f(A)$.

\begin{algorithm}
\caption{Krylov subspace approximation of $f(A+\vb\vb^*)-f(A)$ for Hermitian $A$}
\label{alg:krylovhermitian}
\begin{algorithmic}[1]
\State $\vu_0 = \vnull$.
\State $\vu_1 = (1/\|\vb\|) \vb$.
\State $\beta_1 = 0$.
\For{$j = 1,\dots,m$}
	\State $\vw_j =  A\vu_j - \beta_j\vu_{j-1}$.
	\State $\alpha_j = \vu_j^\ast\vw_j$.
	\State $\vw_j = \vw_j - \alpha_j\vu_j$.
	\State $\beta_{j+1} = \|\vw_j\|$.
	\State $\vu_{j+1} = (1/\beta_{j+1}) \vw_j$.
\EndFor
\State Compute matrix function $X_m(f) = f\big(G_m + \|\vb\|^2 \ve_1 \ve_1^\ast \big) - f( G_m )$
\State Return $U_m X_m(f) U_m^\ast$.
\end{algorithmic}
\end{algorithm}

Algorithm~\ref{alg:krylovhermitian} represents the most basic form of the \emph{Lanczos process}; in particular, it does not employ any reorthogonalization. It is well known that such short recurrences may suffer from severe loss of orthogonality in the presence of round-off errors, so that it can be advisable to use reorthogonalization strategies~\cite{Paige1971,Simon1984}. The most straight-forward to retain numerical orthogonality amongst the basis vectors is to store all basis vectors and perform \emph{full reorthogonalization} in each step of the method.

An alternative is to perform no reorthogonalization at all. In contrast to, e.g., the CG method for linear systems, there are no short recurrences for the iterates $U_mX_m(f)U_m^\ast$ available. In turn, this requires to use a \emph{two-pass} strategy for forming $U_mX_m(f)U_m^\ast$ if one wants to avoid storing the full basis $U_m$. In a first pass, the tridiagonal matrix $G_m$ is assembled (while storing only three basis vectors at a time and discarding the older ones) and $X_m(f)$ is formed. In a second pass, the basis vectors are computed anew (again only storing three at a time), and the parts of interest of $U_mX_m(f)U_m^\ast$ (e.g., its diagonal) are gradually computed. This approach doubles the number of matrix-vector products but reduces the storage requirement from $\calO(mn)$ to $\calO(m^2 + n)$. 

\begin{figure}
\begin{minipage}{.49\textwidth}
\includegraphics[width=\textwidth]{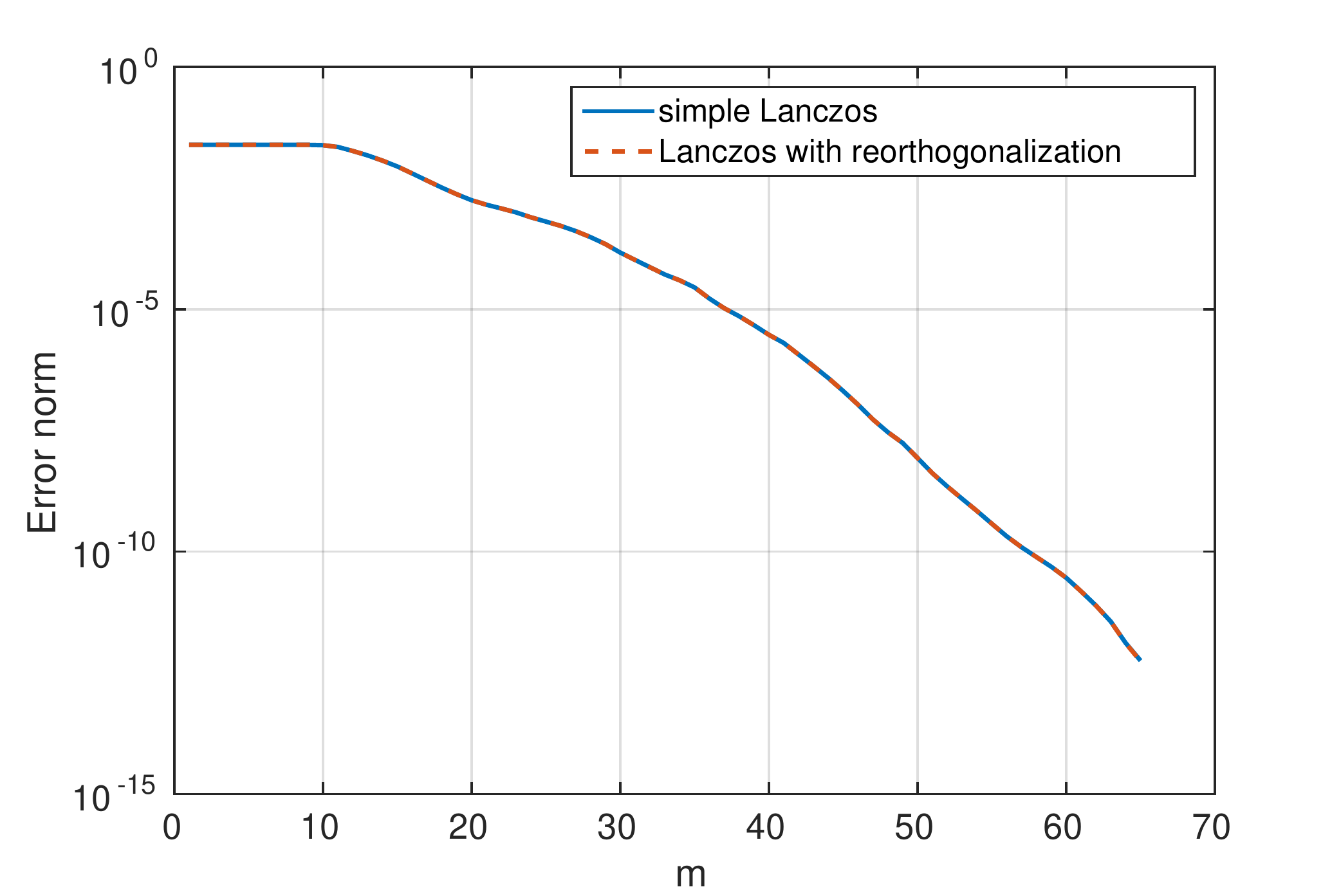}
\end{minipage}
\begin{minipage}{.49\textwidth}
\includegraphics[width=\textwidth]{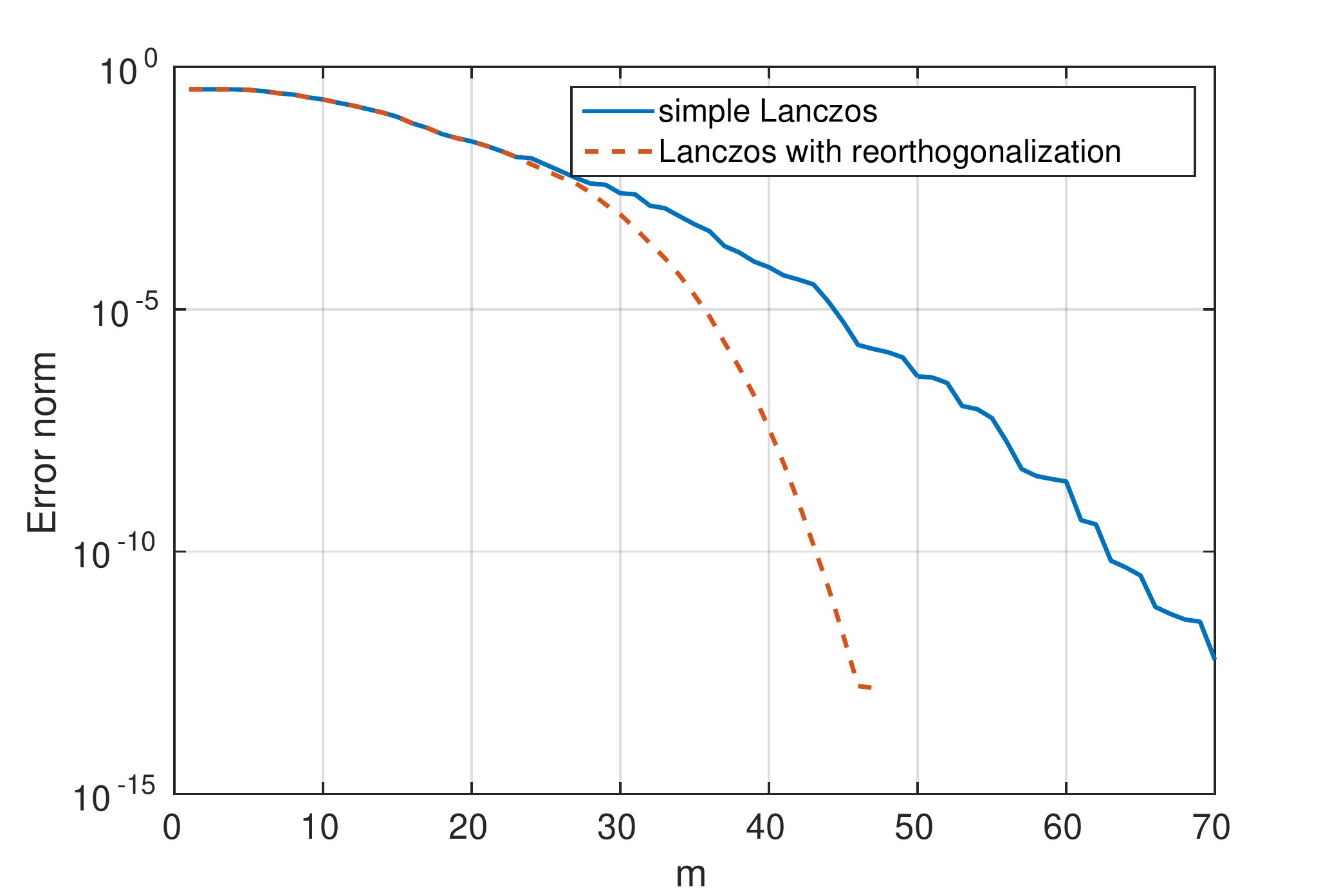}
\end{minipage}
\caption{Convergence curves of simple Lanczos and Lanczos with full reorthogonalization for diagonal matrices with equidistantly spaced (left) and logarithmically spaced (right) eigenvalues in $[10^{-3},10^3]$.\label{fig:reo}}
\end{figure}

\begin{example}
We illustrate the (possible) impact of reorthogonalization on convergence. Consider the diagonal matrices $A_1, A_2 \in \C^{100 \times 100}$, where the eigenvalues of $A_1$ are equidistantly spaced in the interval $[10^{-3},10^3]$ and the eigenvalues of $A_2$ are logarithmically spaced in $[10^{-3},10^3]$. We approximate $f(A_i+\vb\vb^\ast)-f(A_i), i = 1,2$ where $f(z) = \exp(-z)$ and $\vb$ is a random vector of unit norm using Lanczos without reorthogonalization and with full reorthogonalization, as explained above. The resulting convergence curves are shown in Figure~\ref{fig:reo}. For $A_1$, both methods are observed to behave identically, which is well explained by the fact that the eigenvalue distribution of $A_1$ does not favor the convergence of Ritz values, which has a close link to loss of orthogonality~\cite{MeurantStrakos2006,Paige1971}. For $A_2$, the eigenvalue distribution favors the convergence of Ritz values to larger eigenvalues and, in turn, convergence degrades after some time when using no reorthogonalization. Still, the method eventually converges to the same accuracy, it just needs more iterations. This phenomenon is also known from the conjugate gradient method in finite precision arithmetic~\cite{MeurantStrakos2006} and when approximating $f(A)\vb$~\cite{DruskinKnizhnerman1991}. To avoid that orthogonality issues distort our findings, we use the Lanczos method with full reorthogonalization for all numerical experiments in the rest of this paper. \hfill $\diamond$
\end{example}

The computational cost of $m$ steps of Algorithm~\ref{alg:krylovhermitian} is as follows:
\begin{itemize}
\item $m$ matrix-vector products with $A$ (2$m$, when the two-pass approach is used), requiring $\calO(m\cdot\nnz(A))$ operations for a sparse matrix $A$ with $\nnz(A)$ nonzeros;
\item $\calO(m^2n)$ operations for the orthogonalization procedure when using full reorthogonalization (and only $\mathcal{O}(mn)$ when the two-pass approach is used, because in each step orthogonalization against only two previous basis vectors is performed)
\item the computation of $X_m(f)$, which requires the evaluation of two functions of $m \times m$ matrices, which, depending on the function $f$, typically needs at most $\calO(m^3)$ operations. 
\end{itemize}
One should avoid forming $U_mX_m(f)U_m^\ast$ explicitly and we expect that this is actually not needed in most applications involving a large sparse matrix $A$. For example, the computation of communicability measures  discussed in Section~\ref{sec:network} only requires the diagonal entries of $U_mX_m(f)U_m^\ast$, which can be computed directly from $U_m$, $X_m(f)$ with $\calO(m^2 n)$ operations.

\subsection{Non-Hermitian rank-1 updates}

In the general non-Hermitian case, we construct orthonormal bases for the two (polynomial) Krylov subspaces $\calU_m := \spK_m(A,\vb)$
and $\calV_m := \spK_m(A^\ast,\vc)$. Applying the Arnoldi method with reorthogonalization 
results in the Arnoldi relations~\eqref{eq:arnoldi_relationU} and, additionally,
\begin{equation}
A^\ast V_m = V_m H_m + h_{m+1,m} \vv_{m+1}\ve_m^\ast, \label{eq:arnoldi_relationV}
\end{equation}
where the columns of $V_m \in \C^{n\times m}$ form an orthonormal basis of $\calV_m$. Note that both $G_m = U_m^* A U_m$ and $H_m = V_m^\ast A^\ast V_m$ are now $m\times m$ upper Hessenberg matrices. The approximation is chosen in the tensorized subspace $\calU_m \otimes \calV_m$, that is,
\begin{equation}\label{eq:krylov_update_nonhermitian}
 f(A+\vb\vc^*)-f(A) \approx U_m X_m(f) V_m^*
\end{equation}
for some matrix $X_m(f) \in \C^{m \times m}$. Concerning the choice of $X_m(f)$, it turns out that the non-Hermitian analogue of~\eqref{eq:Xm_hermitian_difference} does not have favorable theoretical properties. For example, the polynomial exactness property mentioned above for~\eqref{eq:Xm_hermitian_difference} and proven in section~\ref{sec:exactness} does not hold for such a choice. We will use a different choice, motivated by the following simple result.
\begin{lemma}\label{lem:block}
Let $A \in \Cnn$, and $\vb,\vc \in \Cn$. Define the block matrix
\begin{equation}\label{eq:calA}
\mathcal A := \left[\begin{array}{cc} A & \vb\vc^\ast \\ 0 & A+\vb\vc^\ast\end{array}\right].
\end{equation}
Then
\begin{equation*}
f(\mathcal A) = \left[\begin{array}{cc} f(A) & f(A+\vb\vc^\ast)-f(A) \\ 0 & f(A + \vb\vc^\ast)\end{array}\right].
\end{equation*}
\end{lemma}
\begin{proof}
Letting $\Gamma$ denote a contour that encloses both $\spec(A)$ and $\spec(A+\vb\vc^\ast)$, the Cauchy integral formula~\eqref{eq:contourintro} applied to $f(\mathcal A)$ yields
\begin{eqnarray}
 f(\mathcal A) &=& \frac{1}{2\pi i} \int_\Gamma f(z) (zI-\mathcal A)^{-1} \d z \nonumber \\
 &=& \frac{1}{2\pi i} \int_\Gamma f(z) \begin{bmatrix}(zI-A)^{-1} & -(zI-A)^{-1}\vb\vc^\ast(zI- A - \vb\vc^\ast)^{-1} \\ 0 & (zI-A-\vb\vc^\ast)^{-1} \end{bmatrix} \d z \nonumber \\
 &=& \left[\begin{array}{cc} f(A) & -\frac{1}{2\pi i} \int_\Gamma f(z) (zI- A)^{-1} \vb\vc^\ast (zI-A - \vb\vc^\ast)^{-1} \d z\\ 0 & f(A + \vb\vc^\ast)\end{array}\right]. \label{eq:auxintegral}
\end{eqnarray}
On the other hand, combining the Cauchy integral formula for $f(A+\vb\vc^\ast)$ and $f(A)$ with the second resolvent identity yields
\begin{eqnarray}
f(A+\vb\vc^\ast)-f(A) &=& \frac{1}{2\pi i}\int_\Gamma f(z) \big[ (zI- A - \vb\vc^\ast)^{-1}-(zI-A)^{-1} \big]\d z \nonumber \\
&=& - \frac{1}{2\pi i}\int_\Gamma f(z) (zI- A - \vb\vc^\ast)^{-1} \vb\vc^\ast (zI-A)^{-1} \d z \nonumber \\
&=& - \frac{1}{2\pi i}\int_\Gamma f(z) (zI-A)^{-1} \vb\vc^\ast (zI- A - \vb\vc^\ast)^{-1} \d z, \label{eq:integralformulaupdate}
\end{eqnarray}
which matches the (1,2) block of~\eqref{eq:auxintegral} and thus completes the proof.
\end{proof}

The result of Lemma~\ref{lem:block} shows that the desired update is contained in the (1,2) block of $f(\mathcal A)$. This motivates us to choose $X_m(f)$ as the (1,2) block of $f$ applied to the compression of $\mathcal A$ onto $\calU_m \oplus \calV_m$:
\begin{eqnarray}
 &&f\left( \begin{bmatrix}
          U_m & 0 \\
          0 & V_m
         \end{bmatrix}^\ast \left[\begin{array}{cc} A & \vb\vc^\ast \\ 0 & A+\vb\vc^\ast\end{array}\right]
         \begin{bmatrix}
          U_m & 0 \\
          0 & V_m
         \end{bmatrix}
\right)\nonumber \\
&=&f\left( \begin{bmatrix}
G_m & \|\vb\| \|\vc\| \ve_1 \ve_1^\ast \\ 0 & H_m^* +\|\vc\| V_m^\ast \vb \ve_1^\ast
           \end{bmatrix}  
\right) = \begin{bmatrix}
f(G_m) & X_m (f)\\ 0 & f(H_m^* +\|\vc\| V_m^\ast \vb \ve_1^\ast)
           \end{bmatrix},\label{eq:Xm_nonhermitian}
\end{eqnarray}
where we again assume that $f$ is defined on the spectra of $G_m$ and $H_m^* +\|\vc\| V_m^\ast \vb \ve_1^\ast$. Using Lemma~\ref{lem:block}, it is straightforward to see that this choice of $X_m(f)$ coincides in the Hermitian case with the one from section~\ref{sec:hermitian}.

Algorithm~\ref{alg:krylovnonhermitian} summarizes the proposed procedure, where we omit the algorithmic details for the Arnoldi method for the sake of brevity; see, e.g.,~\cite{GolubVanLoan2013}.
\begin{algorithm}
\caption{Krylov subspace approximation of $f(A+\vb\vc^\ast)-f(A)$}
\label{alg:krylovnonhermitian}
\begin{algorithmic}[1]
\State Perform $m$ steps of the Arnoldi method to compute an orthonormal basis $U_m$ of $\spK_m(A,\vb)$ and $G_m = U_m^\ast A U_m$.
\State Perform $m$ steps of the Arnoldi method to compute an orthonormal basis $V_m$ of $\spK_m(A^*,\vc)$ and $H_m = V_m^\ast A^\ast V_m$.
\State Compute matrix function $F_m = f\left( \begin{bmatrix}
G_m & \|\vb\| \|\vc\| \ve_1 \ve_1^\ast \\ 0 & H_m^* +\|\vc\| V_m^\ast \vb \ve_1^\ast
\end{bmatrix}\right)$.
\State Set $X_m(f) = F_m(1:m,m+1:2m)$.
\State Return $U_m X_m(f) V_m^\ast$.
\end{algorithmic}
\end{algorithm}

The computational effort of Algorithm~\ref{alg:krylovnonhermitian} compares to that of Algorithm~\ref{alg:krylovhermitian} (with full reorthogonalization) as follows. In contrast to the Hermitian case, we now need to build \emph{two} Krylov spaces instead of one, meaning that the number of matrix vector products necessary for $m$ steps of the method increases from $m$ to $2m$. The orthogonalization cost is the same as in the Hermitian case (with full reorthogonalization). 
As the number of operations needed for approximating a (dense) matrix function often grows cubically in the matrix size it is typically about four times as expensive to compute one matrix function of size $2m \times 2m$ than computing two matrix functions of size $m \times m$. The fact that the matrix for which we need to evaluate $f$ in Algorithm~\ref{alg:krylovnonhermitian} is non-Hermitian and not tridiagonal may increase the cost further. However, as long as $m \ll n$, the main cost of Algorithm~\ref{alg:krylovnonhermitian} consists of performing matrix-vector products, so that we can expect it to take roughly twice the computation time of Algorithm~\ref{alg:krylovhermitian} for a problem of the same size and structure.

\begin{remark} \label{remark:higherrank}
There are two different ways of extending our approach from a rank-one modification to a general rank-$k$ modification $f(A+D)$ with $D = BC^\ast$ and $B,C \in \C^{n \times k}$:

1. By letting $\vb_1,\dots,\vb_k$ and $\vc_1,\dots,\vc_k$ denote the columns of $B$ and $C$, respectively, we can write $D$ as a sum of $k$ rank-one matrices, $D = \sum_{i=1}^k \vb_i\vc_i^\ast$, e.g., by computing a singular value decomposition of $D$ and then apply $k$ times Algorithm~\ref{alg:krylovnonhermitian} in order to subsequently incorporate each of the $k$ rank-one modifications. Note that the $i$th step of this procedure requires working with the Krylov subspaces $\spK_m(A+\sum_{j=1}^{i-1} \vb_j\vc_j^\ast, \vb_i)$ and $\spK_m(A^\ast+\sum_{j=1}^{i-1} \vc_j\vb_j^\ast, \vc_i)$ which in general do \emph{not} coincide with $\spK_m(A,\vb_i)$ and $\spK_m(A^\ast,\vc_i)$.

Some care is required in order to make use of Algorithm~\ref{alg:krylovhermitian} 
for a Hermitian rank-$k$ matrix $D$. After a preprocessing step (see, e.g.,~\cite[Section 2.3]{BennerEzzattiKRessnerQuintanaOrtiRemon2011}) one can write $D = \sum_{i=1}^{\tilde k} \vb_i\vb_i^\ast - \sum_{i=\tilde k + 1}^{k} \vb_i\vb_i^\ast$. First, Algorithm~\ref{alg:krylovhermitian} is applied to incorporate the first $\tilde k$ terms followed by a slight modification of this algorithm to incorporate the last $k-\tilde k$ terms

2. \emph{Block Krylov subspaces}~\cite{Gutknecht2006,OLeary1980,SimonciniGallopoulos1995} offer a conceptually different way of dealing with a rank-$k$ modification. Instead of the Arnoldi method, a block Arnoldi method is used for computing orthonormal bases of the block Krylov subspaces $\spK_m(A,B) = \text{range}[B,AB,\ldots, A^{m-1}B]$ and $\spK_m(A^\ast,C)$. The approximation of $f(A+D)-f(A)$ is computed by projecting onto the tensorized block Krylov subspace $\spK_m(A,B) \otimes \spK_m(A^\ast,C)$, leading to a straightforward extension of Algorithm~\ref{alg:krylovnonhermitian}. Block Krylov subspace methods are more complicated to implement, see, e.g.,~\cite{Gutknecht2006} for some of the issues one has to take into account, but they offer (at least) one major advantage. 
Even though the product of $A$ with an $n \times k$ matrix requires the same number of operations as $k$ individual matrix vector products, it often performs much faster on a computer, benefitting from a more ``cache-friendly'' memory access pattern, see, e.g.,~\cite{BakerDennisJessup2006}. Again some care is required in the symmetric case, to derive a block variant of Algorithm~\ref{alg:krylovhermitian}.
\end{remark}

\subsection{Stopping criterion}\label{subsec:stopping}
A simple stopping criterion for Algorithm~\ref{alg:krylovhermitian} or Algorithm~\ref{alg:krylovnonhermitian} is based on the error estimate
\begin{equation}\label{eq:error_estimate_difference}
\|f(A+\vb\vc^\ast)-U_mX_m(f)V_m^\ast\| \approx \|U_{m+d}X_{m+d}(f)V_{m+d}^\ast-U_mX_m(f)V_m^\ast\|
\end{equation}
for some small integer $d\ge 1$. This error estimate is similar in spirit to error estimates for linear systems proposed in~\cite{GolubMeurant1997,Meurant1997}. Evaluating the right-hand side of~\eqref{eq:error_estimate_difference} only requires forming the coefficient matrices $X_m(f), X_{m+d}(f)$ defined in~\eqref{eq:krylov_update_hermitian} or~\eqref{eq:krylov_update_nonhermitian}, because
\begin{eqnarray*}
&& \|U_{m+d}X_{m+d}(f)V_{m+d}^\ast-U_mX_m(f)V_m^\ast\|\nonumber\\
 &=& \left\|U_{m+d}X_{m+d}(f)V_{m+d}^\ast - U_{m+d}\left[\begin{array}{cc} X_m(f) & 0 \\ 0 & 0\end{array}\right]V_{m+d}^\ast\right\| \nonumber\\
&=& \left\|X_{m+d}(f) - \left[\begin{array}{cc} X_m(f) & 0 \\ 0 & 0\end{array}\right]\right\|.
\end{eqnarray*}
We illustrate the behavior of the error estimator~\eqref{eq:error_estimate_difference} by means of a simple numerical experiment.

\begin{figure}
\begin{center}
\includegraphics[width=.7\textwidth]{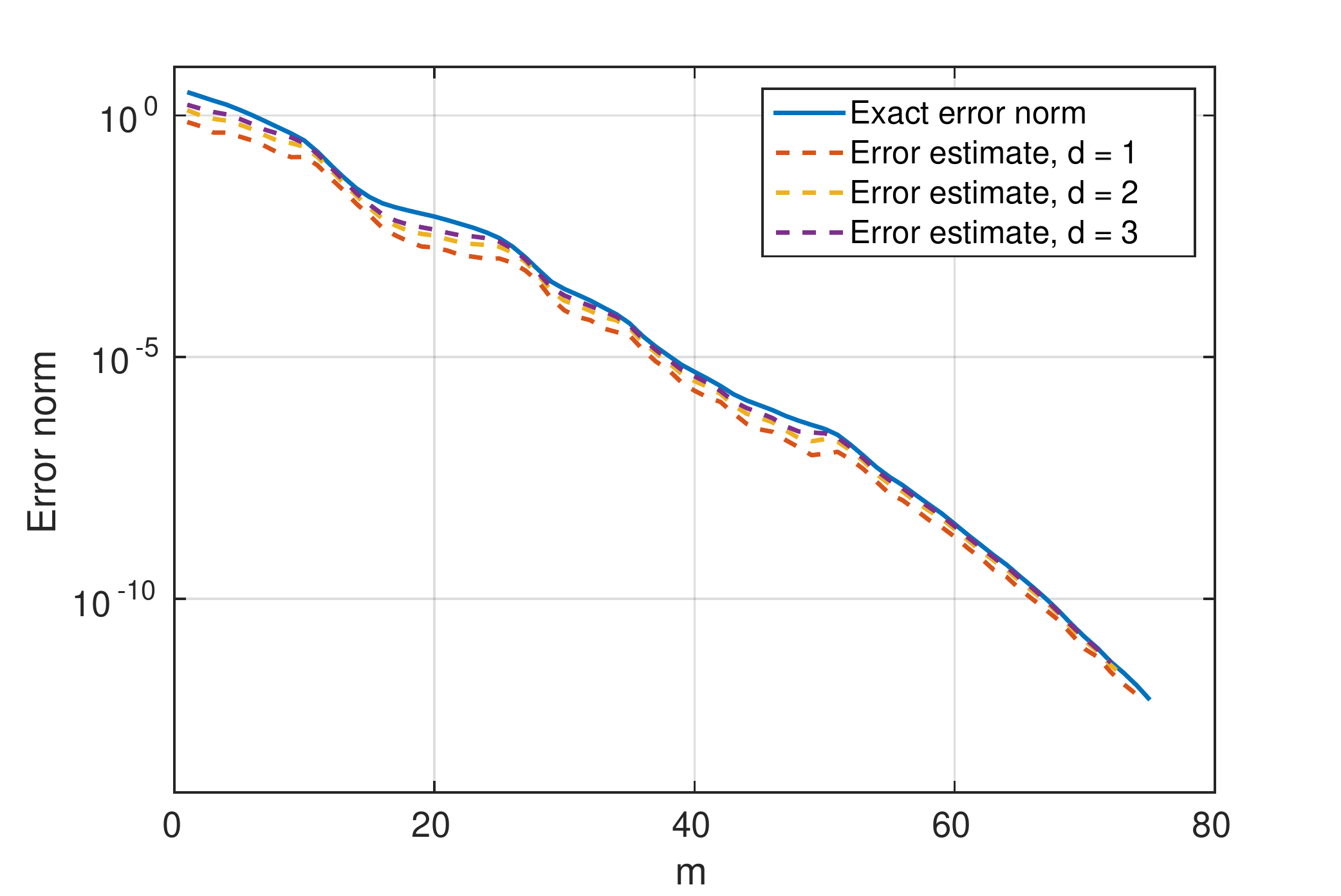}
\end{center}
\caption{Comparison of exact error and difference-based error estimates for the update of the inverse square root described in Example~\ref{ex:error_estimate}.
\label{fig:error_estimates}}
\end{figure}

\begin{example}\label{ex:error_estimate}
We choose $A \in \R^{400 \times 400}$ as the standard finite difference discretization of the two-dimensional Laplace operator on the unit grid, and $\vb,\vc$ as vectors with random, normal distributed entries. We use our proposed Krylov subspace algorithm to approximate $(A+\vb\vc^\ast)^{-1/2} - A^{-1/2}$ and show the convergence history as well as the proposed error estimator for $d = 1,2,3$ in Figure~\ref{fig:error_estimates}. We observe that the difference-based error estimators follow the norm of the exact error very closely and thus give very accurate results already for such small values of $d$. \hfill $\diamond$
\end{example}

We remark that the error estimator~\eqref{eq:error_estimate_difference} is clearly heuristic and can be overly optimistic, especially in situations where the iteration (almost) stagnates, although it works very well in most practical situations. Of course, when already $m+d$ iterations of the method have been performed, one will typically use the approximation $U_{m+d}X_{m+d}V_{m+d}^\ast$ from the $(m+d)$th step instead of the approximation from the $m$th step (for which the error is estimated), as it can be expected to be a more accurate approximation. In fact, in Example~\ref{ex:error_estimate}, for $d=3$ the error of the iterate $X_{m+d}$ lies \emph{below} the error estimate~\eqref{eq:error_estimate_difference} for all but four iterations.

\section{Exactness properties of Krylov subspace approximations for the update}\label{sec:exactness}

In this section, we show that the approximation~\eqref{eq:krylov_update_nonhermitian} is exact when $f$ is a polynomial of degree at most $m$. This serves two purposes: On the one hand, this justifies the particular choice of approximation we made. On the other hand, this provides the fundamental basis for performing our convergence analyses in section~\ref{sec:convergenceanalysis}.

As a tool for proving polynomial exactness, we utilize the following matrix identity, which does not seem to be widely known.

\begin{proposition}\label{pro:matrix_identity}
Let $M, N \in \Cnn$. Then, for any $j \in \N$,
\begin{equation*}
M^j - N^j = \sum_{k = 0}^{j-1} N^{j-1-k}(M-N)M^k.
\end{equation*}
\end{proposition}
\begin{proof}
The proof is by induction on $j$. For $j = 0$, the assertion trivially holds. Suppose now that the assertion holds for $j-1$:
\begin{equation}\label{eq:matrix_identity_proof}
M^{j-1} - N^{j-1} = \sum_{k = 0}^{j-2} N^{j-2-k}(M-N)M^k.
\end{equation}
Multiplying both sides of~\eqref{eq:matrix_identity_proof} by $N$ from the left and adding $M^j$, we find
\begin{equation*}
M^j + NM^{j-1} - N^j = M^j + \sum_{k = 0}^{j-2} N^{j-1-k}(M-N)M^k.
\end{equation*}
Finally, subtracting $M^{j-1}N$ on both sides and noting that $M^j - NM^{j-1} = N^0(M-N)M^{j-1}$ completes the proof.
\end{proof}


\begin{theorem}\label{the:polynomial_exactness}
Let $A\in \Cnn$, $\vb,\vc \in \Cn$. Then the Krylov subspace approximation returned by Algorithm~\ref{alg:krylovnonhermitian} is exact for all $p \in \Pi_{m}$, where $\Pi_{m}$ denotes the space of all polynomials of degree at most $m$, i.e.,
$$p(A+\vb\vc^\ast)-p(A) = U_m X_m(p) V_m^\ast.$$
\end{theorem}
\begin{proof}
By linearity, it suffices to show that the result holds for every monomial $p_j(z) = z^j$, $j = 0,\dots,m$. We recall that $X_m(p_j)$ is the $(1,2)$-block of the matrix
$$\begin{bmatrix}
G_m & \|\vb\|\|\vc\|\ve_1\ve_1^\ast \\
0 & (H_m^\ast + \|\vc\|V_m^\ast\vb \ve_1^\ast)
\end{bmatrix}^j.$$
In particular, $X_m(p_1) = \|\vb\|\|\vc\|\ve_1\ve_1^\ast$. By~\eqref{eq:Xm_nonhermitian}, the result of the theorem is shown if we can establish
\begin{equation} \label{eq:exactpower}
 (A+\vb\vc^\ast)^j - A^j = U_m X_m(p_j) V_m^*.
\end{equation}

Considering 
$$\setlength{\arraycolsep}{1pt}
\begin{bmatrix}
G_m & \|\vb\|\|\vc\|\ve_1\ve_1^\ast \\
0 & (H_m^\ast + \|\vc\|V_m^\ast\vb \ve_1^\ast)
\end{bmatrix}^j\!\!\!\!=\!\! \begin{bmatrix}
G_m^{j-1} & X_m(p_{j-1}) \\
0 & (H_m^\ast + \|\vc\|V_m^\ast\vb \ve_1^\ast)^{j-1}
\end{bmatrix}\!\!\!
\begin{bmatrix}
G_m & X_m(p_1) \\
0 & (H_m^\ast + \|\vc\|V_m^\ast\vb \ve_1^\ast)
\end{bmatrix},$$
we find the relation
\[
X_m(p_j) = G_m^{j-1}X_m(p_1) + X_m(p_{j-1})(H_m^\ast + \|\vc\|V_m^\ast\vb\ve_1^\ast).
\]
Resolving this recursion gives 
\begin{equation}\label{eq:Ej_recursive}
X_m(p_j) = \sum_{k=0}^{j-1} G_m^{j-1-k}(\|\vb\|\|\vc\| \ve_1\ve_1^\ast) (H_m^\ast + \|\vc\|V_m^\ast\vb\ve_1^\ast)^k.
\end{equation}
Recalling that $G_m$ is the compression of $A$ onto $\spK_m(A,\vb)$, a well-known polynomial exactness property (see, e.g.,~\cite{Saad1992,Ericsson1990}) of Krylov subspace approximations yields
\begin{equation}\label{eq:polynomial_exactness_vector1}
\|\vb\| U_m G_m^{\ell} \ve_1 = A^\ell\vb  \text{ for all } \ell = 0,\dots,m-1.
\end{equation}
Similarly, by noting that $H_m + \|\vc\| \ve_1 \vb V_m^\ast$ is the compression of $(A+\vb\vc^\ast)^\ast$ onto
$\spK_m(A^\ast,\vc) = \spK_m((A+\vb\vc^\ast)^\ast,\vc)$, we obtain
\begin{equation}\label{eq:polynomial_exactness_vector2}
\|\vc\| \ve_1^\ast (H_m^\ast+\|\vc\|V_m^\ast\vb\ve_1^\ast)^{\ell}V_m^\ast = \vc^\ast (A+\vb\vc^\ast)^\ell \text{ for all } \ell = 0,\dots,m-1.
\end{equation}
Multiplying~\eqref{eq:Ej_recursive} by $U_m$ from the left and by $V_m^\ast$ from the right then gives
$$U_mX_m(p_j)V_m^\ast = \sum_{k=0}^{j-1} U_mG_m^{j-1-k}(\|\vb\|\|\vc\| \ve_1\ve_1^\ast) (H_m^\ast + \|\vc\|V_m^\ast\vb\ve_1^\ast)^kV_m^\ast$$
which by \eqref{eq:polynomial_exactness_vector1}--\eqref{eq:polynomial_exactness_vector2} gives for $j \leq m$ the relation
$$U_mX_m(p_j)V_m^\ast = \sum_{k=0}^{j-1} A^{j-1-k}\vb\vc^\ast(A+\vb\vc^\ast)^k = (A+\vb\vc^\ast)^j - A^j,$$
where we used Proposition~\ref{pro:matrix_identity} in the second equality. This establishes~\eqref{eq:exactpower} and thus completes the proof.
\end{proof}

\section{Convergence analysis based on polynomial approximation problems} \label{sec:convergenceanalysis}

In this section, we give bounds for the error of the approximations~\eqref{eq:krylov_update_hermitian} and~\eqref{eq:krylov_update_nonhermitian}. The results are based on the polynomial exactness property from Theorem~\ref{the:polynomial_exactness} and connect the approximation quality of~\eqref{eq:krylov_update_hermitian} and~\eqref{eq:krylov_update_nonhermitian} to certain polynomial approximation problems. 

The convergence results in this section rely on a theorem by Crouzeix~\cite{Crouzeix2007,CrouzeixPalencia2017}, for which we first recall some basic concepts. The \emph{field of values} (or \emph{numerical range}) $A \in \Cnn$ is defined as the set
\begin{equation*}
\calW(A) := \{ \vx^\ast\! A \vx : \|\vx\| = 1\},
\end{equation*}
 which is a convex and compact subset of $\C$ containing all eigenvalues of $A$. Further, we make use of the supremum norm 
$$\|f\|_\bbE:= \sup_{z \in \bbE} |f(z)|$$
on a subset $\bbE\subset \C$ for which $f$ is defined. Using these definitions, Crouzeix's theorem states that
\begin{equation}\label{eq:crouzeix}
\|f(A)\| \leq C \|f\|_\bbE
\end{equation}
with a constant $C \leq 1+\sqrt{2}$ for any function $f$ which is analytic in a neighborhood of $\bbE \supseteq \calW(A)$.
Notice that if $A$ is Hermitian then we can compute $\| f(A) \|$ in terms of the eigenvalues of $A$, and thus obviously~\eqref{eq:crouzeix} holds with $C=1$.

We now use~\eqref{eq:crouzeix} together with the polynomial exactness property proven in Theorem~\ref{the:polynomial_exactness} to obtain bounds on the
error at the $m$th step of our method:
\begin{equation}\label{eq:convergence_error}
E_m(f) := f(A+\vb\vc^\ast) - f(A) - U_mX_m(f)V_m^\ast.
\end{equation}
 
\begin{theorem}\label{the:convergence_polynomial_rational}
Let $A$ be Hermitian and let $f$ be defined on a compact set $\bbE$ containing $\calW(A) \cup \calW(A+\vb\vb^\ast)$. Then, the error~\eqref{eq:convergence_error} with $U_m = V_m$ and $X_m(f)$ from~\eqref{eq:Xm_hermitian_difference} satisfies
\begin{equation*}
\|E_m(f)\| \leq 4 \min_{p \in \Pi_m} \|f - p\|_\bbE.
\end{equation*}
\end{theorem}
\begin{proof}
By the polynomial exactness property of $U_mX_m(f)U_m^\ast$ stated in Theorem~\ref{the:polynomial_exactness}, we have
$$E_m(f) = E_m(f) - p(A+\vb\vb^\ast) + p(A) + U_mX_m(p)U_m^\ast = E_m(f-p)$$
for any polynomial $p \in \Pi_m$. For arbitrary $p \in \Pi_m$ we thus have
\begin{eqnarray}\|E_m(f)\| &=& \|E_m(f-p)\| \nonumber\\
&=& \|(f-p)(A+\vb\vb^\ast) - (f-p)(A) - U_mX_m(f-p)U_m^\ast\|\nonumber \\
&\leq& \|(f-p)(A+\vb\vb^\ast)\| - \|(f-p)(A)\| - \|U_mX_m(f-p)U_m^\ast\| \nonumber\\
&\leq& 2\|f-p\|_\bbE+ \|U_mX_m(f-p)U_m^\ast\|, \label{eq:proof_fp}
\end{eqnarray}
where the last inequality follows from~\eqref{eq:crouzeix}. By~\eqref{eq:krylov_update_hermitian} and~\eqref{eq:Xm_hermitian_difference}, we have
$$UX_m(f-p)U_m^\ast = U_m \big((f-p)\big(G_m + \|\vb\|^2 \ve_1 \ve_1^\ast \big) - (f-p)( G_m )\big)U_m^\ast.$$ 
In turn,
\begin{equation}\label{eq:proof_fp2}
\|U_mX_m(f-p)U_m^\ast\| = \|X_m(f-p)\| \leq 2\|f-p\|_\bbE,
\end{equation}
where, using $\calW(G_m) \subseteq \calW(A)$ and $\calW(G_m + \|\vb\|\ve_1\ve_1^\ast) \subseteq \calW(A+\vb\vb^\ast)$, we again applied~\eqref{eq:crouzeix}.  Inserting~\eqref{eq:proof_fp2} into~\eqref{eq:proof_fp} and taking the minimum over all $p \in \Pi_m$ completes the proof.
\end{proof}

Theorem~\ref{the:convergence_polynomial_rational} allows to derive convergence bounds from known polynomial approximation results for analytic functions. The obtained convergence rates will typically be exponential for functions which are analytic in a neighborhood of $\bbE$ and superlinear for entire functions like the exponential.


To extend Theorem~\ref{the:convergence_polynomial_rational} to the non-Hermitian case, we have to assume that $f$ is analytic on the field of values of the block matrix $\calA$ from~\eqref{eq:calA}.
\begin{theorem}\label{the:convergence_polynomial_rational_nonhermitian}
Let $\mathcal A := \left[\begin{array}{cc} A & \vb\vc^\ast \\ 0 & A+\vb\vc^\ast\end{array}\right]$ and let $f$ be analytic in a neighborhood of a compact set $\bbE$ containing $\calW(\mathcal A)$. Then, the error~\eqref{eq:convergence_error}, with $U_m$, $X_m(f)$, and $V_m$ computed by Algorithm~\ref{alg:krylovnonhermitian}, satisfies
\begin{equation*}
\|E_m(f)\| \leq 2C \min_{p \in \Pi_m} \|f - p\|_\bbE
\end{equation*}
with a constant $C \leq 1+\sqrt{2}$.
\end{theorem}
\begin{proof}
By Lemma~\ref{lem:block}, the update $f(A+\vb\vc^\ast) - f(A)$ is the (1,2) block of $f(\mathcal A)$, which we denote by $\big[ f(\mathcal A) \big]_{1,2}$.
Letting $W_m = \left[\begin{array}{cc} U_m & 0 \\ 0 & V_m \end{array}\right]$, we note that the columns of $W_m$ are orthonormal and
\[
 U_mX_m(f)V_m^\ast = [W_m\, f(W_m^\ast\calA W_m)W_m^\ast]_{1,2}
\]
holds by the definition of $X_m(f)$.

As in the proof of Theorem~\ref{the:convergence_polynomial_rational} we use the fact that $E_m(f) = E_m(f-p)$ for any $p \in \Pi_m$. Thus, we obtain for arbitrary $p \in \Pi_m$ that
\begin{eqnarray}
\|E_m(f)\| &=& \|E_m(f-p)\| \nonumber\\
&=& \|(f-p)(A+\vb\vc^\ast) - (f-p)(A) - U_mX_m(f-p)V_m^\ast\|\nonumber \\
&\leq& \|[(f-p)(\calA)]_{1,2}\| + \|[W_m\,(f-p)(W_m^\ast\calA W_m)W_m^\ast]_{1,2}\| \nonumber\\
&\leq& \|(f-p)(\calA)\| + \|(f-p)(W_m^\ast\calA W_m)\| \nonumber\\
&\leq& 2 C\|f-p\|_\bbE \nonumber 
\end{eqnarray}
where the last inequality follows from Crouzeix's theorem, using $\calW(W_m^\ast \calA W_m) \subseteq W(\calA) \subseteq \bbE$.
Taking the minimum over all $p \in \Pi_m$ gives the desired result.
\end{proof}

Note that the matrix $\calA$ from Theorem~\ref{the:convergence_polynomial_rational_nonhermitian} can be easily block-diagonalized:
\[
\calT^{-1} \calA \calT = \begin{bmatrix} A & \\ & A+\vb\vc^\ast\end{bmatrix}, \quad \text{with} \quad \calT = \begin{bmatrix} I & I \\\ 0 & I \end{bmatrix},
\]
with the matrix $\calT$ having the very modest $2$-norm condition number $\kappa(\calT) =  \big( \frac{1+\sqrt{5}}{2} \big)^2$. Unfortunately, this does not seem to admit any meaningful conclusion about the numerical range of $\calA$. In fact, we are not aware of any tight relationship between $\calW(\calA)$ and the numerical ranges of $A$, $A+\vb \vc^*$. Writing
$$\calA = \begin{bmatrix} A & 0 \\ 0 & A \end{bmatrix} + \begin{bmatrix} 0 & \vb\vc^\ast \\0  & \vb\vc^\ast\end{bmatrix} = \begin{bmatrix} A & 0 \\ 0 & A \end{bmatrix} + \begin{bmatrix} \vb \\ \vb \end{bmatrix} [\vnull^\ast, \vc^\ast]$$
we obtain from~\cite[Section 1.0.1 \& Property 1.2.10]{HornJohnson1991} the inclusion
\begin{equation*}
\calW(\calA) \subseteq \calW(A) + \calW(\vu\vv^\ast), \quad \vu =\begin{bmatrix} \vb \\ \vb \end{bmatrix},\, \vv = \begin{bmatrix} \vnull \\ \vc \end{bmatrix}
\end{equation*}
where $+$ refers to the Minkowski sum of sets. In general, the field of values of a rank-one matrix $\vu\vv^\ast$ is an ellipse with focal points $0$ and $\vv^\ast\vu$ and minor semi-axis $\frac{1}{2} (\|\vu\|^2 \, \|\vv\|^2 - |\vv^\ast\vu|^2)^{1/2}$, which in our special case amounts to focal points $0$ and $\vc^\ast\vb$ and minor semi-axis $\frac{1}{2}(2\|\vb\|^2\, \|\vc\|^2 - |\vc^\ast\vb|^2)^{1/2}$.

In Section~\ref{sec:convdiff}, we will illustrate for an example of practical relevance that $\calW(\calA)$ can have rather undesirable properties, to the extent that Theorem~\ref{the:convergence_polynomial_rational_nonhermitian} is of little help in understanding the convergence of our algorithms. In the next section, we therefore derive convergence bounds by an approach that avoids the dependence on $\calW(\calA)$ and only depends on $\calW(A)$ and $\calW(A+\vb\vc^\ast)$. While the second set may still be larger than $\calW(A)$, it is at least always smaller than $\calW(\calA)$.

\section{Convergence results based on integral representations}\label{sec:convergence_integral}

We begin by considering results based on the Cauchy integral formula in Section~\ref{subsec:cauchy} and then focus on the special case of Markov functions in Section~\ref{subsec:markov}.

\subsection{Convergence results based on the Cauchy integral formula}\label{subsec:cauchy}

Let $f$ be analytic on a domain $\bbE$ containing the eigenvalues of $A$ and $A+\vb\vc^\ast$. We recall~\eqref{eq:integralformulaupdate}:
\begin{equation}
f(A + \vb\vc^\ast) - f(A) = -\frac{1}{2\pi i}\int_\Gamma f(z) (zI-A)^{-1}\vb\vc^\ast(zI - A  - \vb\vc^\ast)^{-1}\d z\label{eq:update_integral},
\end{equation}
with $\Gamma = \partial\bbE$. The integrand in~\eqref{eq:update_integral} involves solutions of shifted linear systems. Letting
\begin{equation}\label{eq:shifted_systems} 
(zI - A)\vx(z) = \vb \text{ and } (zI - A -  \vb\vc^\ast)^\ast \vy(z) = \vc,
\end{equation}
we can compactly write the right-hand side of~\eqref{eq:update_integral} as
\begin{equation}\label{eq:integral_compact}
f(A + \vb\vc^\ast) - f(A) = \frac{1}{2 \pi i}\int_\Gamma f(z) \vx(z)\vy(z)^\ast \d z.
\end{equation}
Now, consider the FOM~\cite{Saad1981} approximations for~\eqref{eq:shifted_systems}, given by
\begin{eqnarray*}
\vx_m(z) &:=& \|\vb\| U_m (zI-G_m)^{-1}\ve_1, \nonumber\\
 \vy_m(z) &:=& \|\vc\|V_m(\bar z I - H_m - \|\vc\| \ve_1 \vb^\ast V_m)^{-1}\ve_1,
\end{eqnarray*}
where we used the Arnoldi decompositions~\eqref{eq:arnoldi_relationU} and~\eqref{eq:arnoldi_relationV}. 
Recalling that $X_m(f)$ is defined as the (1,2) block of 
$$f \left( \begin{bmatrix}
G_m &  \|\vb\|\|\vc\| \ve_1\ve_1^\ast \\
0 & H_m^\ast + \|\vc\|V_m^\ast\vb\ve_1^\ast
\end{bmatrix} \right)$$
and using contour integration, as in the proof of Lemma~\ref{lem:block}, we find that
\begin{equation} \label{eq:krylov_projection_integral}
U_mX_m(f) V_m^\ast = - \frac{1}{2\pi i} \int_\Gamma f(z) \vx_m(z)\vy_m(z)^\ast \d z
\end{equation}

In other words, the approximation to the update matrix can be interpreted as the integral over outer products of FOM approximations for the shifted linear systems~\eqref{eq:shifted_systems}.

Inserting~\eqref{eq:integral_compact} and~\eqref{eq:krylov_projection_integral} into the definition~\eqref{eq:convergence_error} of the error gives
\begin{eqnarray}
E_m(f) &=& -\frac{1}{2\pi i} \int_\Gamma f(z) \big(\vx(z)\vy(z)^\ast -  \vx_m(z)\vy_m(z)^\ast\big) \d z \nonumber\\
 &=& -\frac{1}{2 \pi i} \int_\Gamma f(z) \big( \vx(z)(\vy(z)-\vy_m(z))^\ast + (\vx(z) - \vx_m(z))\vy_m(z)^\ast \big) \d z. \label{eq:Em_integral}
\end{eqnarray}
Convergence estimates can now be obtained by taking norms in the contour integral~\eqref{eq:Em_integral} and bounding all occurring terms. To do so, we first introduce some notation. In the following, $\bbE$ is a compact, convex set containing both $\calW(A)$ and $\calW(A+\vb\vc^\ast)$. The closed unit disk is denoted by $\bbD$. Denoting by $\barC := \C \cup \{\infty\}$ the extended complex plane, let $\psi$ be the conformal mapping from $\barC \setminus \bbD$ onto $\barC \setminus \bbE$, normalized at infinity such that $\psi(\infty) = \infty$, $\psi^\prime(\infty) > 0$. Note that $\psi^\prime(w)$ exists for almost all $w$ on $\partial \bbD$. We are now in the position to state the following norm bounds on the quantities involved in~\eqref{eq:Em_integral}. 

\begin{lemma}\label{lem:norms}
For all $m \geq 2$ and $z=\psi(u)\in \C \setminus \mathbb E$ we have 
\begin{eqnarray} &&\label{eq:norm}
    \max\Bigl\{ \frac{\| \vx(z) \|}{\| \vb \|} , \frac{\| \vx_m(z) \|}{\| \vb \|}, \frac{\| \vy(z) \|}{\| \vc \|} , \frac{\| \vy_m(z) \|}{\| \vc \|}
   \Bigr\} \leq \frac{1}{\dist(z,\mathbb E)} 
   \leq          \frac{|u|/\psi'(\infty)}{(|u|-1)^2},
    \\ && \label{eq:error}
    \max\Bigl\{  \frac{ \| \vx(z) - \vx_m(z) \| }{\| \vb  \|} ,
    \frac{ \| \vy(z) - \vy_m(z) \| }{\| \vc  \|}  \Bigr\}
    \leq \frac{4|u|^{1-m}}{|\psi'(u)|(|u|^2-1)} \leq \frac{4|u|^{-m}}{\dist(z,\mathbb E)} ,
\end{eqnarray}
where $\dist(z,\bbE)$ denotes the distance of $z$ from $\bbE$.
\end{lemma}
\begin{proof}
   The first inequality in~\eqref{eq:norm} follows for $\vx(z)$ immediately from the fact that $$
      \| \vx(z) \| = \| (zI-A)^{-1} \vb \| \leq \| (zI-A)^{-1} \| \, \|\vb \| \leq \frac{\| \vb \|}{\dist(z,\mathbb E)} , 
   $$
   and analogously for $\vx_m(z)$, $\vy(z)$, and $\vy_m(z)$. For the second inequality in~\eqref{eq:norm}, we recall a result by K\"uhnau (see Theorem 3.1 in \cite{TohTrefethen1999} and its proof):
\begin{equation} \label{eq:level}
       \frac{|u|-1}{1+1/|u|} \leq \frac{\dist(\psi(u),\mathbb E)}{|\psi'(u)|}\leq \frac{|u|^2-1}{|u|},
    \end{equation}
    which, in fact, does not require $\bbE$ to be convex.
By convexity of $\bbE$, we also have the inequality
    \begin{equation}\label{eq:golusin}
      \Big| \frac{\psi'(u)}{\psi'(\infty)} - 1 \Big| \leq \frac{1}{|u|^2}
    \end{equation}
    due to Gr\"otzsch and Golusin; see \cite[Section 2]{KovariPommerenke1967}. Combining these two inequalities leads to
    $$
         \frac{1}{\dist(z,\mathbb E)} \leq 
         \frac{1+1/|u|}{|\psi'(u)|(|u|-1)}
         \leq 
         \frac{|u|}{\psi'(\infty)(|u|-1)^2},
    $$
    showing \eqref{eq:norm}.
    
We now turn to the second set \eqref{eq:error} of inequalities. We will make use of properties of the Faber transform $\mathcal F$ of a function analytic in the open unit disk and continuous on the closed unit disk; see~\cite{Ellacott1983,BeckermannReichel2009} for more details. Letting $G(v)= \frac{1}{\psi'(u)(u-v)}$ we obtain
   \begin{eqnarray} \label{eq:definitiong}
        g(\zeta):=\mathcal F(G)(\zeta) &=& \frac{1}{2\pi i} \int_{\partial \mathbb E} G(\psi^{-1}(\widetilde \zeta)) \frac{d\widetilde \zeta}{\widetilde \zeta-\zeta} \\
        &=&\frac{1}{2\pi i} \int_{|v[=1} G(v) \frac{\psi'(v) \d v}{\psi(v)-\zeta} = \frac{1}{\psi(u)-\zeta}, \nonumber
   \end{eqnarray}
   where the last equality follows from the residue theorem; see also~\cite[Eqns (2.5), (2.7)]{BeckermannReichel2009}.
   Let $P$ be defined by the formula
   $$
       P(v) = \frac{\overline u}{\psi'(u)(|u|^2-1)}
       \Bigl( \frac{v}{u}\Bigr)^{m-1}
       + \frac{1}{\psi'(u)(v-u)}
       \Bigl( \Bigl( \frac{v}{u}\Bigr)^{m-1} - 1 \Bigr),
   $$
   depending on the parameter $u$. Then it is straightforward to verify that $P$ is a polynomial of degree at most $m-1$, and that
   $$
        G(v)-P(v) = \frac{1}{\psi'(u)(|u|^2-1)} \frac{v \overline u - 1}{u-v} \Bigl( \frac{v}{u}\Bigr)^{m-1}
   $$ vanishes at $0$ because $m \ge 2$. Thus, with the notation of \cite[\S 2]{BeckermannReichel2009}, and $p=\mathcal F(P)\in \Pi_{m-1}$
   we get that $\mathcal F(G-P)=\mathcal F_+(G-P)=g-p$, using~\eqref{eq:definitiong}. Now Theorem 2.1 in~\cite{BeckermannReichel2009}, being related to the fundamental work of Crouzeix~\cite{Crouzeix2007,CrouzeixPalencia2017}, implies
   $$
        \| (zI -A)^{-1} - p(A)\| =
        \| (g-p)(A) \| \leq 2 \, \max_{|v|=1} | G(v)-P(v) |
        = \frac{2|u|^{1-m}}{|\psi'(u)|(|u|^2-1)} .
   $$
   Since $W(G_m)\subset W(A)$, the same upper bound is obtained for
    $\| (zI-G_m)^{-1} - p(G_m) \|$. Because of $p \in \Pi_{m-1}$ and the exactness property~\eqref{eq:polynomial_exactness_vector1}, we have that $p(A)\vb = \|\vb\| V_m p(G_m)\ve_1$ and thus
 \begin{eqnarray*}
       \frac{ \| \vx(z) - \vx_m(z) \| }{\| \vb  \|}  &=& \frac{ \| \vx(z) - p(A)\vb - (\vx_m(z) - \|\vb\| V_m p(G_m)\ve_1) \| }{\| \vb  \|}  \\
&\leq& \| (zI-A)^{-1} - p(A) \| + \| (zI-G_m)^{-1} - p(G_m) \| \\
       &\leq& \frac{4|u|^{1-m}}{|\psi'(u)|(|u|^2-1)} ,
    \end{eqnarray*}
    that is, we have shown the first inequality of \eqref{eq:error}. 
    The second inequality follows from a combination with the second inequality in \eqref{eq:level}. The inequalities in~\eqref{eq:error} involving $\vy(z)$ instead of $\vx(z)$ are proven in a completely analogous fashion.
\end{proof}

In order to state our first explicit convergence bound, consider the (compact) level sets $\mathbb E_r$, which are defined via their complements as $\bbE_r = \C \setminus \bbE_r^c$, where $\bbE_r^c : = \{z \in \C \setminus \bbE : |\psi^{-1}(z)| > r\}$.

\begin{theorem}\label{the:cauchy}
Suppose that $f$ is analytic in $\bbE_R$. Then, for $1< r < R$
    $$
        \| E_m(f) \| \leq R^{-m-1} 
        \frac{16/\psi'(\infty)}{(1-r/R)(1-1/r)^3} \, \max_{z \in \Gamma_R} |f(z)| \, \| \vb \| \, \| \vc \|,
    $$
		where $\Gamma_R = \partial \bbE_R$.
\end{theorem}
\begin{proof}
Let $\Gamma=\partial \bbE_r$. Then we have, by taking norms in~\eqref{eq:Em_integral},
\begin{equation}\label{eq:proof_cauchy1}
\|E_m(f)\| \leq \frac{1}{2 \pi} \int_\Gamma |f(z)| \big( \|\vx(z)\|\|\vy(z)-\vy_m(z)\| + \|\vy_m(z)\|\|\vx(z) - \vx_m(z)\| \big) |\d z|.
\end{equation}
Because of the polynomial exactness of the Krylov approximation~\eqref{eq:krylov_update_nonhermitian} proven in Theorem~\ref{the:polynomial_exactness}, the error of our Krylov approximation is the same for $f$ and $f-p$, for any $p \in \Pi_m$. This allows us to conclude from~\eqref{eq:proof_cauchy1} that
\begin{equation}\label{eq:proof_cauchy2}
\|E_m(f)\| \leq \frac{\|f-p\|_\Gamma}{2 \pi} \int_\Gamma \big( \|\vx(z)\|\|\vy(z)-\vy_m(z)\| + \|\vy_m(z)\|\|\vx(z) - \vx_m(z)\| \big) |\d z|.
\end{equation}
Applying the substitution $z=\psi(u)$ with $|u|=r$, we obtain from Lemma~\ref{lem:norms} the upper bound 
\begin{eqnarray}
&&\int\limits_{|u|=r}\!\!\!\!\big( \|\vx(\psi(u))\|\|\vy(\psi(u))-\vy_m(\psi(u))\| + \|\vy_m(\psi(u))\|\|\vx(\psi(u)) - \vx_m(\psi(u))\| \big) \frac{|\psi^\prime(u)\d u|}{2\pi}\nonumber\\
&&\leq \|\vb\|\|\vc\|\int_{|u|=r} \frac{8 r^{2-m}}{\psi'(\infty)(r^2-1)(r-1)^2} \frac{|\d u|}{2\pi} \leq \frac{8 r^{-m-1}}{\psi'(\infty)(1-1/r)^3} \|\vb\|\|\vc\|\label{eq:proof_cauchy_integral_bound}
\end{eqnarray}
for the integral in~\eqref{eq:proof_cauchy2}. 
Further, we can use a partial Faber sum (cf.~\cite[Remark 3.3]{BeckermannReichel2009} with $\rho = R/r$) to find the Bernstein-type estimate for the best polynomial approximation on the convex set $\bbE_r$
\begin{equation}\label{eq:proof_cauchy_approx_bound}
\min_{p\in \Pi_m} \| f - p \|_{\Gamma} 
         \leq \left(\frac{r}{R}\right)^m \frac{2r}{R-r} \, \max_{z \in \Gamma_R} |f(z)|.
\end{equation} 
Inserting~\eqref{eq:proof_cauchy_integral_bound} and~\eqref{eq:proof_cauchy_approx_bound} into~\eqref{eq:proof_cauchy2} then yields the desired result.
\end{proof}

Let us illustrate Theorem~\ref{the:cauchy} for some particular sets $\mathbb E$ and the particular case of the exponential function $f(z)=\exp(z)$. We suppose in the following that $\mathbb E$ is convex and symmetric with respect to the real axis. Then $\psi(R)\in \mathbb R$ is the element in $\mathbb E_R$ with the largest real part, and hence
$$
        R^{-m-1} \max_{z\in \Gamma_R} | f(z) | = e^{\psi(R)}/{R^{m+1}} .
$$
Our aim will be to choose $R>1$ to make the above right-hand side small. This task has been (implicitly) accomplished by Hochbruck and Lubich \cite[Section~3]{HochbruckLubich1997} for various families of sets like real and purely imaginary intervals, disks, and wedge-shaped sets with a corner at $\psi(1)$, see also the analysis of \cite[Section~4]{BeckermannReichel2009} where also general convex domains with corners have been considered. 

Our upper bounds will be stated in terms of $\psi(1)$, the element of $\mathbb E$ of largest real part, and in terms of $\rho = \psi'(\infty)$, the logarithmic capacity of $\mathbb E$ (which is increasing as the set becomes larger). For the four families of sets mentioned above, explicit formulas are known for $\psi(R)$ in terms of $\psi(1),\rho$ and $R$.

We start with some negative result. By convexity, the function $r \mapsto r\psi'(r)$ is known to be increasing. Provided that
$      \psi'(1) \geq m+1   $, elementary calculus implies that $R\mapsto 
e^{\psi(R)}/{R^{m+1}}$ is increasing for $R\in [1,\infty)$, and thus Theorem~\ref{the:cauchy} is not useful in this case. However, for $m+1 \gg 2 \rho \geq \psi'(1)$, we get the following result of superlinear convergence. 
\begin{corollary}\label{cor:exp_superlinear}
    Let $\mathbb E$ be convex and symmetric with respect to the real axis.
    Then, under the conditions of Theorem~\ref{the:cauchy} and $f(z)=\exp(z)$, we have for $m+1 \geq e \rho$ 
    \begin{equation}\label{eq:bound_exp_superlinear}
        \| E_m(f) \| \leq \frac{672}{\rho} e^{\psi(1)} \Big( \frac{\rho e}{m+1} \Bigr)^{m+1}
        \, \| \vb \| \, \| \vc \|. 
    \end{equation}
\end{corollary}
\begin{proof}
  From the Gr\"otzsch and Golusin inequality \eqref{eq:golusin}, we find for $1 \leq r \leq R$
	\begin{equation}\label{eq:golusin_r}
	\psi^\prime(r) - \rho\Big(1 + \frac{1}{r^2}\Big) \leq 0, \text{ where } \rho = \psi^\prime(\infty).
	\end{equation}
	Integrating~\eqref{eq:golusin_r} from $1$ to $R$ then yields
   $$
      \psi(R) \leq \psi(1) + \rho \Big(R-\frac{1}{R} \Big).
  $$
  By arguing as in the proof of~\cite[Theorem~4]{HochbruckLubich1997} we find for $R=\frac{m+1}{\rho}$ that $e^{\psi(R)}/R^{m+1}\leq e^{\psi(1)}(\frac{e\rho}{m+1})^{m+1}$. 
  In addition, since $(1-1/\sqrt{R})^{-4} \le (1-e^{-1/2})^{-4} \le 42$, we conclude that 
  $$
       \frac{16/\rho}{(1-1/\sqrt{R})^4} \frac{e^{\psi(R)}}{R^{m+1}}
       \leq \frac{672}{\rho} e^{\psi(1)} \Big( \frac{\rho e}{m+1} \Bigr)^{m+1},
  $$
  and our claim follows from Theorem~\ref{the:cauchy} with $r=\sqrt{R}$.
\end{proof}

\begin{example}\label{ex:exp_herm}
\begin{figure}
\begin{center}
\includegraphics[width=.7\textwidth]{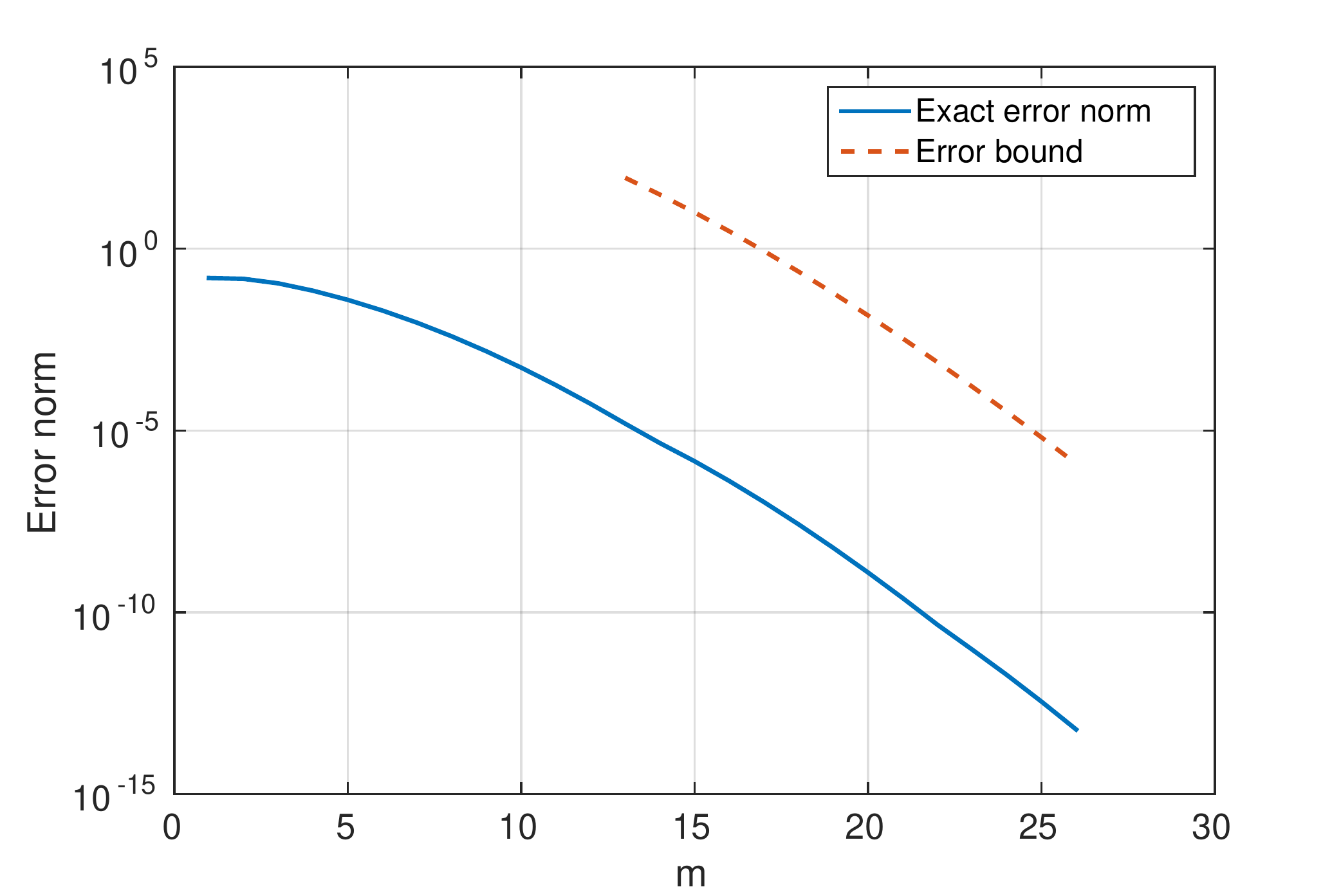}
\end{center}
\caption{Exact error norm and error bound~\eqref{eq:bound_exp_superlinear} for the update of the matrix exponential described in Example~\ref{ex:exp_herm}.\label{fig:exp_herm}}
\end{figure}
We illustrate the result of Corollary~\ref{cor:exp_superlinear} by a simple numerical experiment. We choose $A \in \C^{100 \times 100}$ as a diagonal matrix with eigenvalues equidistantly spaced in $[-20,0]$ and $\vb$ as a random vector of unit norm, resulting in $\spec(A-\vb\vb^\ast) \subseteq [-20.2,0] =: \bbE$. We apply Algorithm~\ref{alg:krylovhermitian} (modified to account for the update $-\vb\vb^\ast$) for approximating $\exp(A-\vb\vb^\ast)-\exp(A)$ and report the resulting convergence curve, together with the error bound~\eqref{eq:bound_exp_superlinear} in Figure~\ref{fig:exp_herm}. We note that the convergence rate is predicted very accurately, but that the magnitude of the error is severly overestimated due to the large constant in~\eqref{eq:bound_exp_superlinear}, which we expect to not be optimal. We thus expect that in general the convergence slope of our bound will be quite accurate, while there is a (large) constant distance between the actual convergence curve and the error bound, something which happens quite frequently for bounds based on the field of values. For practical purposes, we thus suggest to ignore the constant in~\eqref{eq:bound_exp_superlinear}.

We remark that by combining the result of Theorem~\ref{the:convergence_polynomial_rational} with the result of~\cite[Corollary 4.1]{BeckermannReichel2009} on polynomial approximation of the exponential function, we find a bound that predicts essentially the same convergence rate as the bound~\eqref{eq:bound_exp_superlinear}, albeit with a much smaller constant. The real use of the integral representation is thus in the non-Hermitian case, where it allows to circumvent the reliance on $\calW(\calA)$. \hfill $\diamond$
\end{example}

In view of known results for disks (see~\cite[Example after Theorem~5]{HochbruckLubich1997}), we do not expect $E_{m}(f)$ to be small in general for $m+1<\psi'(1)$.
However, in the case of a corner at $\psi(1)$, the right-most element of $\mathbb E$, we have a different regime of convergence for $m+1 \ll \rho$. Here we consider only the wedge-like set $\mathbb E=\mathbb E(\alpha,\rho)$ with
\begin{equation}\label{eq:wedge}
    \psi(w) = \psi(1) + \rho w \Big(1-\frac{1}{w}\Big)^\alpha, \quad 1<\alpha \leq 2,
\end{equation}
having an outer angle of $\alpha\pi$ at $\psi(1)$, by slightly improving \cite[Theorem~6]{HochbruckLubich1997} (though following \cite[Corollay~4.2]{BeckermannReichel2009} one could include more general $\mathbb E$ having such an angle). 
Notice that $\mathbb E(2,\rho)=[\psi(1)-4\rho,\psi(1)]$ is a real interval of capacity $\rho$.

\begin{corollary}\label{cor:wedge}
    Consider the wedge-like set $\mathbb E=\mathbb E(\alpha,\rho)$ for $\alpha\in (1,2]$ and $\rho>0$.
    Then, under the conditions of Theorem~\ref{the:cauchy} and $f(z)=\exp(z)$, we have for $m+1-4/\alpha\in [\alpha \rho^{1/\alpha},\alpha \rho]$
    \begin{equation}\label{eq:bound_exp_wedge}
        \| E_m(f) \| \leq \frac{(4\rho^{1/\alpha})^4}{\rho} 
        \exp \Bigl( \psi(1) - (\alpha-1) \Bigl( \frac{m+1-\frac{4}{\alpha}}{ \alpha \rho^{1/\alpha}} \Bigr)^{\frac{\alpha}{\alpha-1}}  \Bigr)\, \| \vb \| \, \| \vc \|.
    \end{equation}
\end{corollary}
\begin{proof}
    We consider for $R > 1$ the strictly increasing function
    $$
        u = u(R) = \Bigl( \frac{\psi(R)-\psi(1)}{\rho} \Bigr)^{1/\alpha}
                 = R^{1/\alpha}- R^{-1+1/\alpha} ,
    $$
    and notice that $u\leq 1$ implies that $R^{1/\alpha} = u + R^{-1+1/\alpha}\leq u+1\leq 2$. Hence, provided that $u \in (0,1)$,
    \begin{eqnarray*}
          \frac{e^{\psi(R)-\psi(1)}}{R^{m+1}(1-1/\sqrt R)^4}
          &\leq &
          \frac{16 e^{\psi(R)-\psi(1)}}{R^{m+1}(1-1/R)^4}
          \\
          &\leq& \frac{16}{u^4} \exp \Big( \rho u^\alpha + \Big(m+1-\frac{4}{\alpha}\Big) \log\Big(\frac{1}{R}\Big) \Big)
          \\
          &\leq& \frac{16}{u^4} \exp \Big( \rho u^\alpha + \Big(m+1-\frac{4}{\alpha}\Big) \Big(\frac{1}{R}-1\Big) \Big)
          \\
          &= & \frac{16}{u^4} \exp \Bigl( \rho u^\alpha - \Big(m+1-\frac{4}{\alpha}\Big) \frac{u}{R^{1/\alpha}} \Bigr)
          \\
          &\leq& \frac{16}{u^4} \exp \Bigl( \rho u^\alpha - \Big(m+1-\frac{4}{\alpha}\Big) \frac{u}{2} \Bigr).
    \end{eqnarray*}
    The argument of the exponential function on the right takes its minimum at 
    $$
           u^{\alpha-1} = \frac{m+1-4/\alpha}{2\alpha\rho} \in [\rho^{-\frac{\alpha-1}{\alpha}} , 1]
    $$
    by our assumption on $m$.
    Hence
    \begin{eqnarray*}
          \frac{e^{\psi(R)-\psi(1)}}{R^{m+1}(1-1/\sqrt R)^4}
          &\leq& \frac{16}{u^4} \exp \Bigl( - (\alpha-1) \Big(m+1-\frac{4}{\alpha}\Big) \frac{u}{2\alpha} \Bigr)
          \\
          &\leq& 16 \rho^{4/\alpha}  \exp \Bigl( - (\alpha-1) \Bigl( \frac{m+1-\frac{4}{\alpha}}{2\alpha \rho^{1/\alpha}} \Bigr)^{\frac{\alpha}{\alpha-1}}  \Bigr),
    \end{eqnarray*}
  and our claim follows from Theorem~\ref{the:cauchy}.
\end{proof}

\begin{example} \label{ex:exp_wedge}
\begin{figure}
\begin{center}
\includegraphics[width=.7\textwidth]{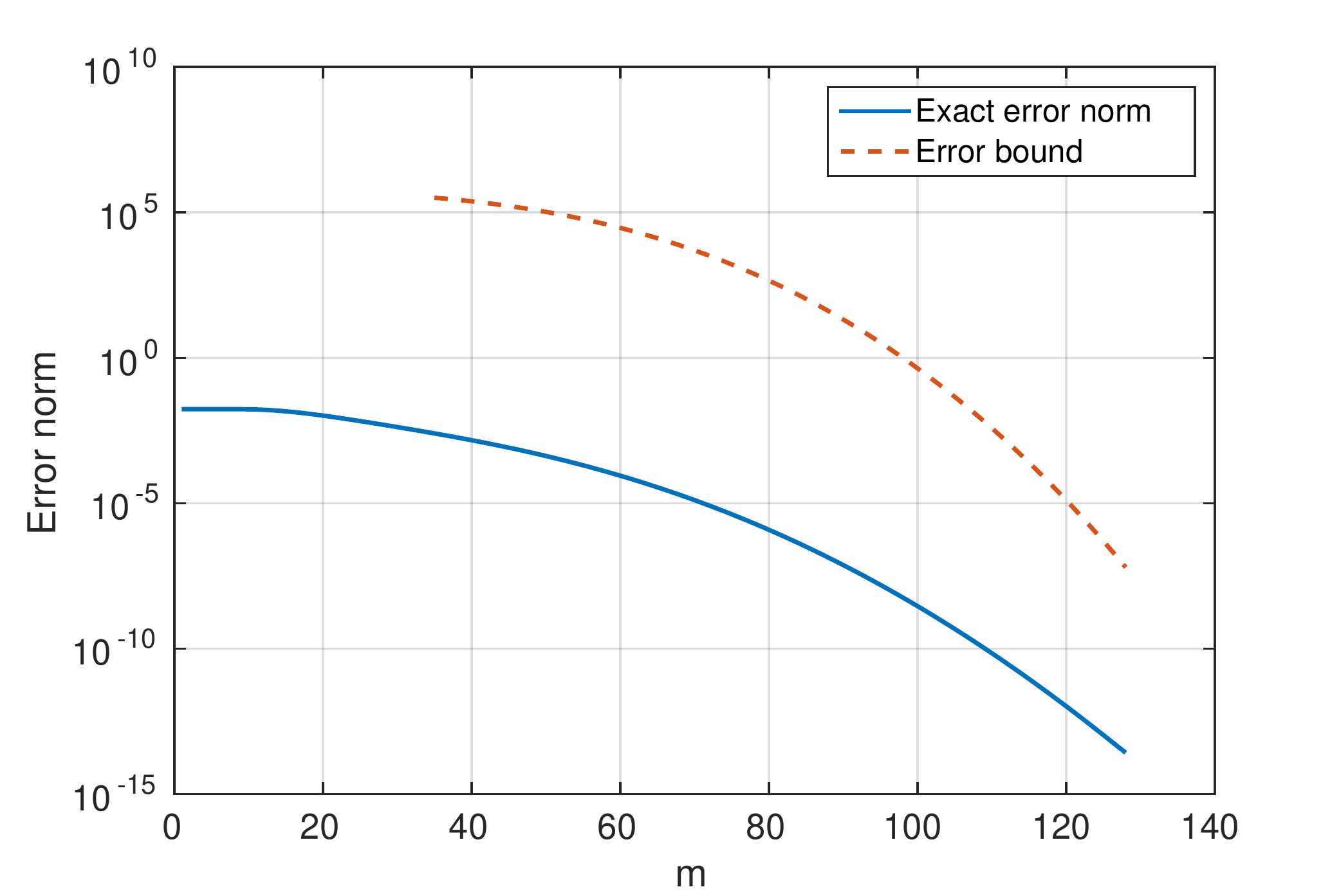}
\end{center}
\caption{Exact error norm and error bound~\eqref{eq:bound_exp_wedge} for the update of the matrix exponential described in Example~\ref{ex:exp_wedge}.\label{fig:exp_wedge} }
\end{figure}
We illustrate the result of Corollary~\ref{cor:wedge} for a diagonal matrix $A \in \C^{1000 \times 1000}$ with eigenvalues in the wedge-like set~\eqref{eq:wedge} with $\psi(1) = 0$, $\alpha = 3/2$ and $\rho = 100$. The vector $\vb$ is again chosen as a random vector of unit norm, which results in $\calW(A-\vb\vb^\ast) \subseteq \bbE$ where $\bbE$ is a wedge-like set corresponding to $\psi(1) = 0$, $\alpha = 1.5$ and $\rho = 101$. We apply Algorithm~\ref{alg:krylovnonhermitian} (modified to account for the update $-\vb\vb^\ast)$ for approximating $\exp(A-\vb\vb^\ast)-\exp(A)$ and report the resulting convergence curve together with the error bound~\eqref{eq:bound_exp_wedge} in Figure~\ref{fig:exp_wedge}. We again observe that the error norm is overestimated by a few orders of magnitude, while the convergence slope is predicted quite well (although not as accurately as in the Hermitian case in Example~\ref{ex:exp_herm}).\hfill $\diamond$
\end{example}

In the particular case of symmetric data, $A=A^*$ and $\vc=\vb$, the smallest set containing both $\calW(A)$ and $\calW(A+\vb\vb^\ast)$ is the interval $\bbE=[\lambda_{\min}(A),\lambda_{\max}(A+\vb\vb^*)]$, where $\lambda_{\min}(\cdot)$ and $\lambda_{\max}(\cdot)$ denote the smallest and largest eigenvalues of a Hermitian matrix. For the exponential function a refined analysis would be possible, using the fact that the coefficients of the Chebyshev series of $f(z)=\exp(z)$ are explicitly known in terms of Bessel functions. We believe, however, that we do not gain much qualitatively compared to the results of the two preceding corollaries with the choices
$$    \psi(1) = \lambda_{\max}(A+\vb\vb^*), \quad 
    \rho = \frac{\lambda_{\max}(A+\vb\vb^*)-\lambda_{\min}(A)}{4}, \quad \alpha=2.$$

\subsection{Convergence results for Markov functions}\label{subsec:markov}
The approach outlined above extends to other integral representations of $f$. In particular, a \emph{Markov function} can be written as
\begin{equation}\label{eq:markov}
f(x) = \int_\alpha^\beta \frac{\d\mu(z)}{x-z},
\end{equation}
where $\mu$ is a positive measure with support in the interval $[\alpha,\beta]$ with $-\infty \leq \alpha < \beta < \infty$. Any such Markov function is analytic in $\barC \setminus [\alpha,\beta]$. Examples of Markov functions include inverse fractional powers
$$
    f(z)=z^{-\gamma} = \frac{\sin(\gamma\pi)}{\pi} \int_{-\infty}^0 \frac{(-x)^{-\gamma}\d x}{z-x} 
$$
for $\gamma \in (0,1)$ or the logarithm
$$
    f(z)=\frac{1}{z}\log(1+z) = \int_{-\infty}^{-1} \frac{(-1/x)\d x}{z-x},
$$
see, e.g.,~\cite{BergForst1975,Henrici1977} for more details and other examples of Markov functions.

Combining the result of Lemma~\ref{lem:norms} with the integral representation of the error~\eqref{eq:Em_integral} now allows to obtain convergence estimates.

\begin{theorem}\label{the:markov}
Let $\bbE$ be symmetric with respect to the real axis and let $\omega\in \bbE \cap \R$ denote the element of $\bbE$ with smallest real part. For a Markov function $f$ with $\alpha < \beta <\omega$ in~\eqref{eq:markov}, it holds that
$$\| E_m(f) \| \leq \frac{8 \, | f'(\omega) | }{|\phi(\beta)|^m} \, \| \vb \| \, \| \vc \|,$$    
where $\phi$ is mapping conformally from $\barC \setminus\bbE$ onto $\barC \setminus \bbD$.
\end{theorem}
\begin{proof}
Let us first show that we may take $\Gamma=[\alpha,\beta]$ in equation \eqref{eq:Em_integral}. Take $r>1$ such that $f$ is analytic in $\mathbb E_r$, that is, $[\alpha,\beta]\cap \mathbb E_r=\emptyset$. Since all expressions
$\vx(z),\vy(z)^*,\vx_m(z),\vy_m(z)^*$ are analytic outside $\mathbb E$ and decay like $1/z$ at $\infty$, we get by exchanging integration and using the Cauchy residual theorem 
\begin{eqnarray*} 
     E_m(f)&=&\frac{1}{2\pi i} \int_{\partial \mathbb E_r} \int_{\alpha}^\beta \frac{\d\mu(t)}{t-z} \big( \vx(z)(\vy(z)-\vy_m(z))^\ast + (\vx(z) - \vx_m(z))\vy_m(z)^\ast \big) \d z
     \\&=& \int_{\alpha}^\beta \big( \vx(t)(\vy(t)-\vy_m(t))^\ast + (\vx(t) - \vx_m(t))\vy_m(t)^\ast \big) \d\mu(t) . 
\end{eqnarray*}
By taking norms, we obtain
\begin{equation}
\|E_m(f)\| \leq \int_\alpha^\beta ( \|\vx(t)\|\|\vy(t)-\vy_m(t)\| + \|\vy_m(t)\| \|\vx(t) - \vx_m(t))\|) \dmu \label{eq:Em_norm}
\end{equation}
Applying the first inequality in~\eqref{eq:norm} and the second inequality in~\eqref{eq:error} to the individual terms in~\eqref{eq:Em_norm} then yields
\begin{equation}\label{eq:Em_norm2}
\|E_m(f)\| \leq 8\|\vb\|\|\vc\|\int_\alpha^\beta \frac{|u|^{-m}\d\mu(t)}{\dist(t,\bbE)^2}.
\end{equation}
Using the fact that $t = \psi(u)$, i.e., $u = \phi(t)$ and the fact that $\dist(t,\bbE) = \omega - t$ for $t\in [\alpha,\beta]$, the inequality~\eqref{eq:Em_norm2} yields
\begin{equation*}
\|E_m(f)\| \leq 8\|\vb\|\|\vc\|\int_\alpha^\beta \frac{1}{|\phi(t)|^{m}}\frac{\d\mu(t)}{(\omega - t)^2}.
\end{equation*}
Since the function $1/|\phi(t)|$ is monotonically increasing on $[\alpha,\beta]$ (see, e.g., the proof of Theorem 6.1 in~\cite{BeckermannReichel2009}), we further have 
\begin{equation}\label{eq:Em_norm4}
\|E_m(f)\| \leq \frac{8\|\vb\|\|\vc\|}{|\phi(\beta)|^m}\int_\alpha^\beta \frac{\d\mu(t)}{(\omega - t)^2}.
\end{equation}
The result of the theorem now follows by noting that the integral on the right-hand side of~\eqref{eq:Em_norm4} is exactly $|f^\prime(\omega)|$, see, e.g.,~\cite{AlzerBerg2002}.
\end{proof}

We now show how the bound from Theorem~\ref{the:markov} simplifies when $\bbE$ is an ellipse on the right of $[\alpha,\beta]$ (or, as a special case, an interval). This gives a more explicit idea of the convergence behavior one can expect.

\begin{corollary}
Let $f$ be a Markov function~\eqref{eq:markov} with $-\infty \leq \alpha < \beta < \infty$, let $\bbE$ be an ellipse
\begin{equation}\label{eq:ellipse}
\bbE = \{ z \in \C : | z - \sigma + \tau | + |z - \sigma - \tau| \leq \tau (\rho + \rho^{-1})\},
\end{equation}
where $\sigma \in \R,\, \tau > 0,\, \rho \geq 1$. Let 
\begin{equation}\label{eq:omega}
\omega := \sigma - \frac{\tau}{2}(\rho+\rho^{-1}).
\end{equation}
If $\beta < \omega$ and $\bbE$ is symmetric to the real axis we have
\begin{equation}\label{eq:convergence_bound_markov_ellipse}
\|E_m(f)\| \leq 8 \, |f'(\omega)| \, \|\vb\|\|\vc\| \left(\rho\cdot\frac{\sqrt{\kappa}-1}{\sqrt{\kappa}+1}\right)^m,\quad \kappa = \frac{|\beta-\sigma|+\tau}{|\beta-\sigma|+\tau}.
\end{equation}
\end{corollary}
\begin{proof}
Obviously, $\omega$ defined in~\eqref{eq:omega} is the element of smallest real part in $\bbE \cap \R$. The conformal mapping $\phi$ for $\bbE$ is given by the inverse Joukowski mapping
\begin{equation*}
\phi(z) = \rho^{-1} \left(\zeta + \sqrt{\zeta^2-1}\right), \quad \zeta = \frac{z - \sigma}{\tau}.
\end{equation*}
By straight-forward algebraic manipulations, we find that
\begin{equation*}
\frac{1}{|\phi(\beta)|} = \rho \frac{\sqrt{\kappa}-1}{\sqrt{\kappa}+1}
\end{equation*}
with $\kappa$ as defined in~\eqref{eq:convergence_bound_markov_ellipse}. Applying Theorem~\eqref{the:markov} then gives the desired result.
\end{proof}

A further simplification is possible in the Hermitian positive definite case, where the ellipse $\bbE$ is an interval. We assume $\beta \leq 0$ and bound $1/|\phi(\beta)| \leq 1/|\phi(0)|$ in the following result, as this gives a result which closely resembles the classical convergence bound for the conjugate gradient method.

\begin{corollary}\label{cor:markov_Hpd}
Let $f$ be a Markov function~\eqref{eq:markov} with $-\infty \leq \alpha < \beta \leq 0$, let $A$ be Hermitian positive definite and let $\vb = \vc^\ast$. Further, let 
$$\kappa_\ast = \frac{\lmax(A+\vb\vb^\ast)}{\lmin(A)}$$. We then have
\begin{equation}\label{eq:convergence_bound_markov_Hpd}
\|E_m(f)\| \leq 8 \, |f^\prime(0)| \, \|\vb\|^2 \left(\frac{\sqrt{\kappa_\ast}-1}{\sqrt{\kappa_\ast}+1}\right)^m.
\end{equation}
\end{corollary}
\begin{proof}
In the Hermitian positive definite case, $\calW(A) = [\lmin(A),\lmax(A)]$ and $\calW(A+\vb\vb^\ast) = [\lmin(A+\vb\vb^\ast),\lmax(A+\vb\vb^\ast)]$. As $\lmin(A) \leq \lmin(A+\vb\vb^\ast)$ and $\lmax(A) \leq \lmax(A+\vb\vb^\ast)$, we can thus take
$$\bbE = [\lmin(A), \lmax(A+\vb\vb^\ast)].$$
This is a special case of an ellipse~\eqref{eq:ellipse} with
\begin{equation}\label{eq:rhosigmatau}
\rho = 1,\quad\sigma = \tau = \frac{\lmax(A+\vb\vb^\ast)-\lmin(A)}{2}.
\end{equation}
Noting that $\lmin(A)$ is the smallest element in $\bbE \cap \R = \bbE$ and inserting~\eqref{eq:rhosigmatau} into~\eqref{eq:convergence_bound_markov_ellipse} gives the desired result after some simple calculations.
\end{proof}

\begin{example} \label{ex:invsqrt}
\begin{figure}
\begin{center}
\includegraphics[width=.7\textwidth]{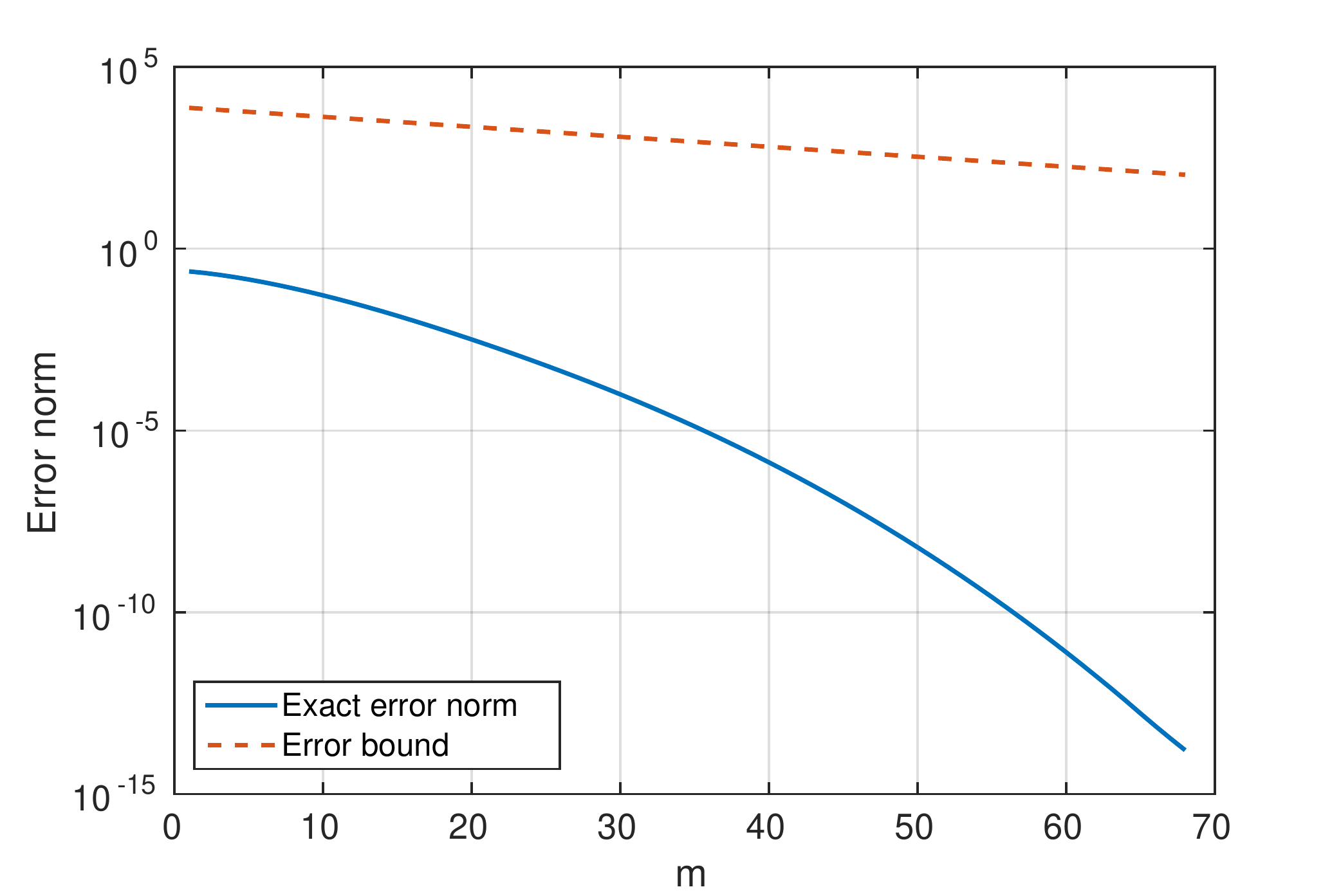}
\end{center}
\caption{Exact error norm and error bound~\eqref{eq:convergence_bound_markov_Hpd} for the update of the matrix inverse square root described in Example~\ref{ex:invsqrt}.\label{fig:invsqrt}}
\end{figure}%
To illustrate the result of Corollary~\ref{cor:markov_Hpd}, consider a diagonal matrix $A \in \R^{100 \times 100}$ with equidistantly spaced eigenvalues in the interval $[0.1,10]$. As in the previous examples, we choose $\vb$ as a random vector of unit norm, which results in $\spec(A+\vb\vb^\ast) \subseteq [0.1,10.1] =: \bbE$. We approximate $(A+\vb\vb^\ast)^{-1/2}-A^{-1/2}$ and report the resulting convergence curve, together with the error bound~\eqref{eq:convergence_bound_markov_Hpd} in Figure~\ref{fig:invsqrt}. We can observe the typical shortcoming of bounds of the form~\eqref{eq:convergence_bound_markov_Hpd} for Markov functions; they only predict linear convergence, while superlinear convergence due to spectral adaption can be observed in this case, see, e.g.,~\cite{BeckermannGuettel2012}. In addition, as we also observed for the exponential function, the large constant in the bound~\eqref{eq:convergence_bound_markov_Hpd} leads to a overestimation of the order of magnitude of the error.\hfill $\diamond$
\end{example}

\begin{remark}
We make two remarks concerning Corollary~\ref{cor:markov_Hpd}.
\begin{enumerate}
\item Similar to what we pointed out already for the exponential function, if we use known approximation results for Markov functions, like, e.g., \cite[Theorem 6.1 \& Remark 6.3]{BeckermannReichel2009} together with Theorem~\ref{the:convergence_polynomial_rational}, we obtain the same convergence factor (with a different constant) as in Corollary~\ref{cor:markov_Hpd} in the Hermitian case.
\item It is possible to obtain a slightly sharper version of~\eqref{eq:convergence_bound_markov_Hpd} in which the constant $8$ is replaced by $4$, and $\kappa_\ast$ is replaced by the maximum of the Euclidean norm condition numbers of $A$ and $A+\vb\vb^\ast$ by explicitly exploiting the connection to the conjugate gradient method~\cite{HestenesStiefel1952} and integrating over the classical CG convergence bounds, see, e.g.,~\cite[Lemma 4.1 \& Theorem 4.3]{FrommerGuettelSchweitzer2014} for a similar technique. We refrain from giving the details of this, as the improvement over the bound~\eqref{eq:convergence_bound_markov_Hpd} is quite marginal in most cases.
\end{enumerate}
\end{remark}

\section{Applications and numerical experiments}\label{sec:experiments}
In this section, we present possible applications for the developed methods, with a special focus on the computation of communicability measures in network analysis. All experiments are performed in MATLAB R2016b on a Linux machine with Intel Core i7-6700 CPU and 32 GB main memory.

\subsection{Updating network communicability measures} \label{sec:network}
Matrix functions, especially the matrix exponential, play an important role in network analysis, see~\cite{EstradaHigham2010} and the references therein.

Given an undirected graph $G = (V,E)$ with $V = \{1,\ldots,n\}$ and $E\subseteq V\times V$, we let $A \in \R^{n \times n}$ be the adjacency matrix of $G$ defined by $a_{ij} = a_{ji} = 1$ if $(i,j) \in E$ and $a_{ij} = 0$ otherwise. Note that $A$ is symmetric. 
Introduced in~\cite{Estrada2005}, the \emph{subgraph centrality} of node $i$ is given by
\begin{equation} \label{eq:subcentrality}
\frac{[\exp(A)]_{ii}}{\trace(\exp(A))}.
\end{equation}
 This quantity represents a weighted average of the number of closed walks which connect node $i$ to itself, see, e.g.,~\cite{EstradaHigham2010} for details. An interesting question is how the subgraph centralities in a graph change when an edge is added to or removed from the graph; see for example~\cite{ArrigoBenzi2016}, where the effect of adding/removing edges on the so-called \emph{total communicability} is studied. The addition of an edge $(i,j)$, with $i\not=j$, corresponds to the rank-two modification
\begin{equation}\label{eq:adjacency_update}
A + BC^\ast, \quad \text{where} \quad B = [\ve_i, \ve_j], \quad C = [\ve_j, \ve_i],
\end{equation}
of the adjacency matrix. Analogously, the removal of an edge $(i,j)$ corresponds to a  rank-two modification. In turn, these tasks fit perfectly into the framework considered in this paper. 

In our experiments:
\begin{itemize}
 \item Following~\cite{BenziEstradaKlymko2013,BenziBoito2010}, the diagonal entries of $\exp(A)$, needed for evaluating~\eqref{eq:subcentrality} for the original matrix $A$, are estimated by combining the Lanczos method with quadrature~\cite{BenziGolub1999,GolubMeurant2010}. Specifically, we use the implementation of these methods from the \texttt{mmq} toolbox~\cite{MeurantToolbox}.
 \item To handle the rank-two modification~\eqref{eq:adjacency_update} we perform two consecutive rank-one updates with Algorithm~\ref{alg:krylovhermitian}, as outlined in Remark~\ref{remark:higherrank}. Note that while $B \neq C$, the resulting matrix $BC^\ast$ is Hermitian. Thus, we can perform a preprocessing step in order to obtain two Hermitian rank-one updates and then apply Algorithm~\ref{alg:krylovhermitian}. As said before, it is not necessary to explicitly form the matrix $U_mX_m(\exp) V_m^\ast$, we only evaluate its diagonal entries.
\end{itemize}

\begin{table}
\caption{Number $n$ of nodes and number $k$ of edges of the networks used in our experiments, together with the computation time for updating the subgraph centrality of all nodes, when modifying $10$ edges, using Algorithm~\ref{alg:krylovhermitian} vs. recomputing them from scratch using the \texttt{mmq} toolbox.\label{tab:networks}}
\begin{center}
\begin{tabular}{lrrrr}
\hline
Network & $n$ & $k$ & Algorithm~\ref{alg:krylovhermitian} & \texttt{mmq} \\
\hline
Gleich/Minnesota & 2,642 & 6,606 & 0.15 s & 1.80 s\\
Pajek/Erdos992 & 6,100 & 15,030  & 0.62 s & 6.22 s\\
Pajek/USpowerGrid & 4,941 & 13,188 & 0.48 s & 4.96 s \\
SNAP/ca-HepTh & 9,877 & 51,971 & 1.25 s & 14.58 s\\
SNAP/email-Enron & 36,692 & 367,662 & 1.96 s& 147.47 s \\
\hline
\end{tabular}
\end{center}
\end{table}

We now suppose that the diagonal of the matrix exponential has been computed beforehand (in an expensive \emph{offline} calculation). We then add or remove edges from the graph and compare the computation time needed for updating them using Algorithm~\ref{alg:krylovhermitian} to the time needed for recomputing them from scratch using the \texttt{mmq} toolbox. When using the \texttt{mmq} toolbox, we perform five Lanczos iterations per diagonal entry (this value is, e.g., also used in~\cite{BenziBoito2010}) and can be expected to give a rather rough approximation of the exact value. When using Algorithm 1, we aim for an accuracy of $10^{-6}$ according to the stopping criterion from Section~\ref{subsec:stopping}, evaluated with $d = 2$.
We observed that typically between 10 and 30 iterations of the algorithm are necessary to fulfill this criterion.
We use five different networks available from the SuiteSparse Matrix Collection~\cite{SuiteSparse}, and for each network we randomly choose ten edges that are added to or removed from the network (i.e., in order to compute the update, we have to call Algorithm~\ref{alg:krylovhermitian} a total of $20$ times). Table~\ref{tab:networks} summarizes the obtained results. For all networks under consideration we observe that our algorithm clearly outperforms recomputation. 

\begin{figure}
\begin{center}
\includegraphics[width=.65\textwidth]{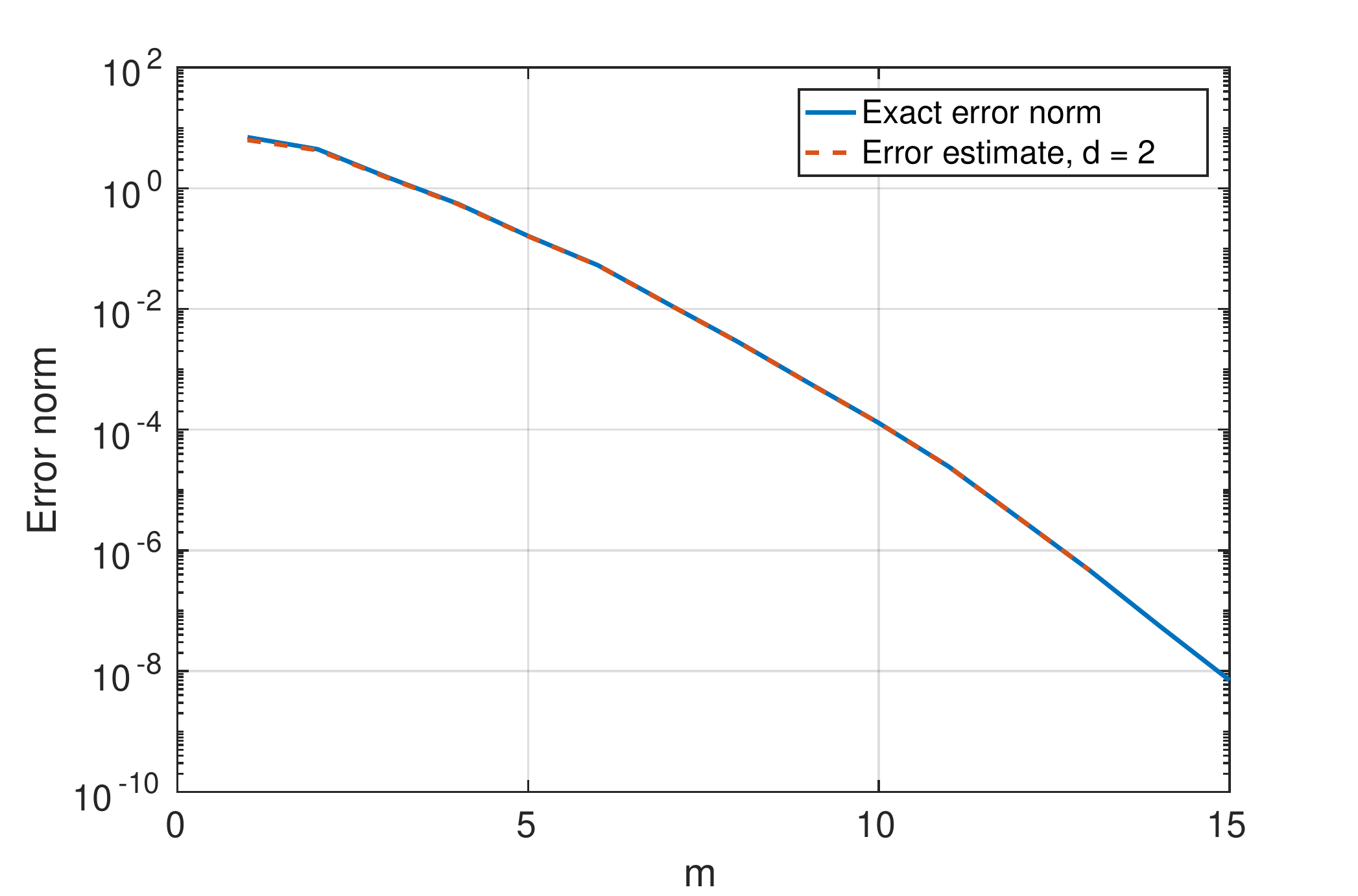}
\end{center}
\caption{Convergence curve of Algorithm~\ref{alg:krylovnonhermitian} for computing a rank-one update of the matrix \texttt{Pajek/USpowerGrid} together with the error estimate~\eqref{eq:error_estimate_difference}.\label{fig:convergence_powergrid}}
\end{figure}

To also illustrate how the convergence profile of the method looks like and to judge the quality of the error estimate~\eqref{eq:error_estimate_difference} in a practical situation, we show the convergence curve of the first update performed for the network \texttt{Pajek/USpowerGrid} together with the error estimate in Figure~\ref{fig:convergence_powergrid}.
We can observe that the method converges very smoothly and that the error estimate is almost identical to the exact error norm for all iterations. 

\subsection{One-dimensional convection diffusion equation}\label{sec:convdiff}

\begin{figure}
\begin{minipage}{.49\textwidth}
\includegraphics[width=\textwidth]{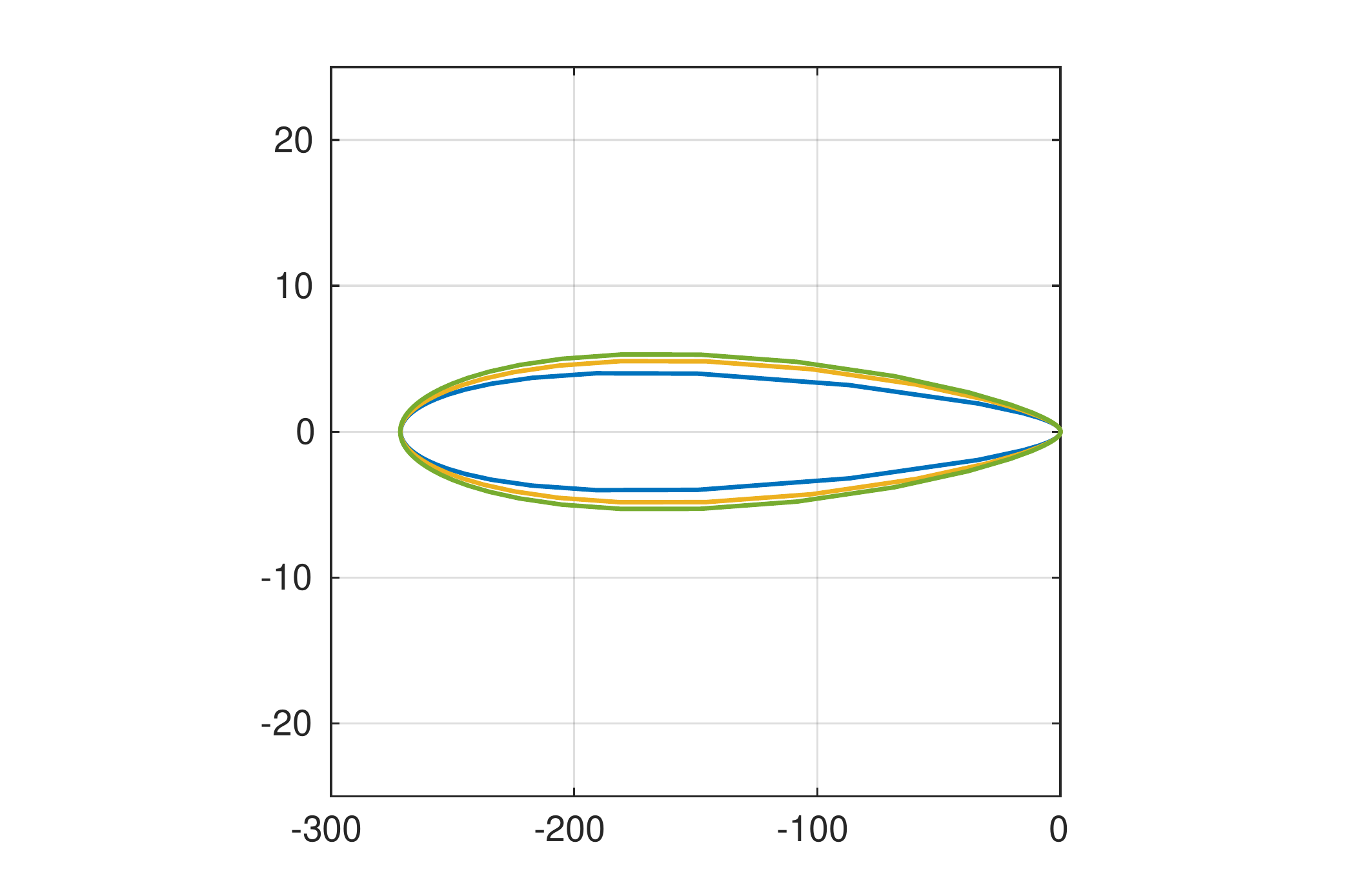}
\end{minipage}
\begin{minipage}{.49\textwidth}
\includegraphics[width=\textwidth]{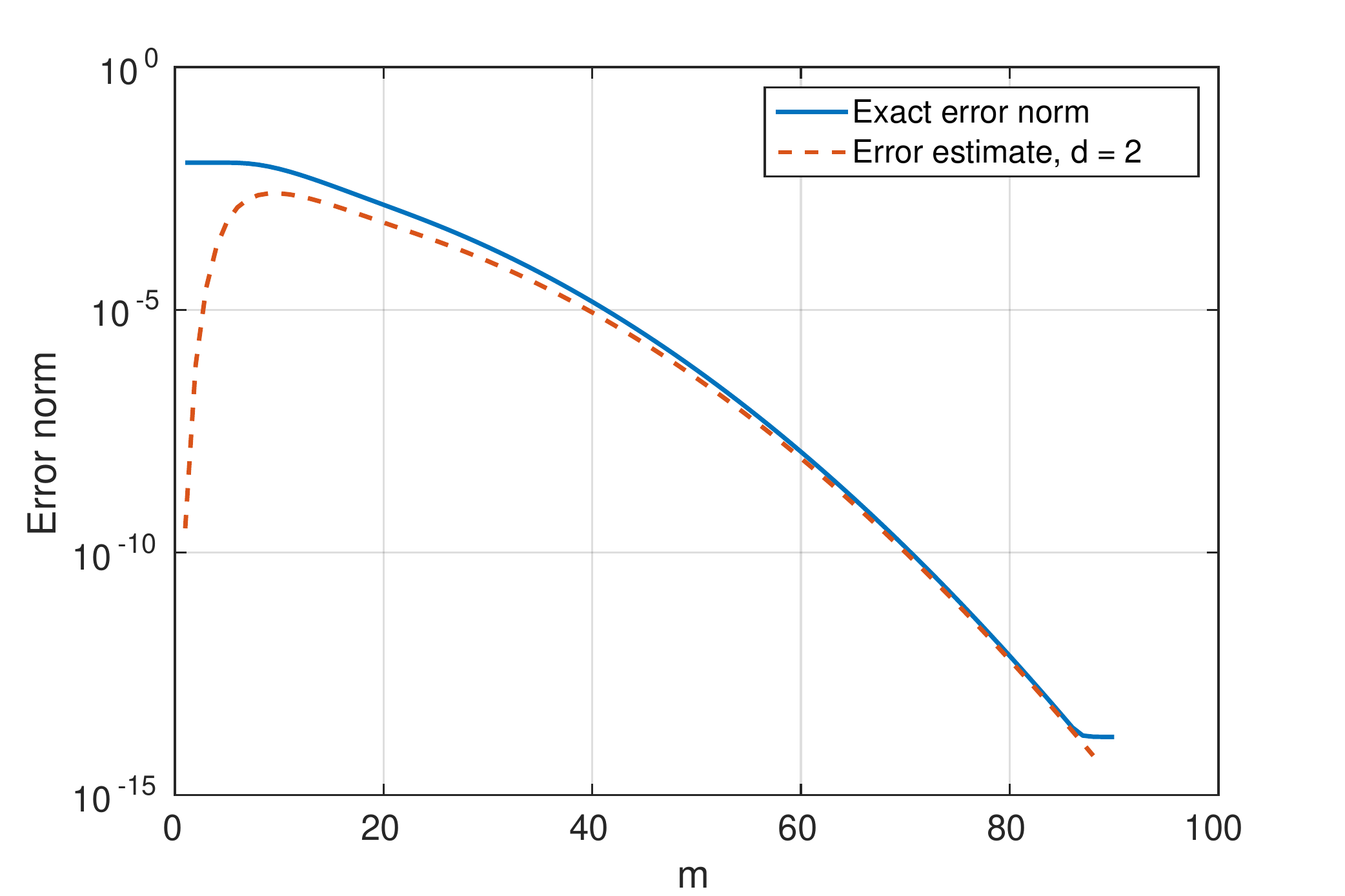}
\end{minipage}
\begin{minipage}{.49\textwidth}
\includegraphics[width=\textwidth]{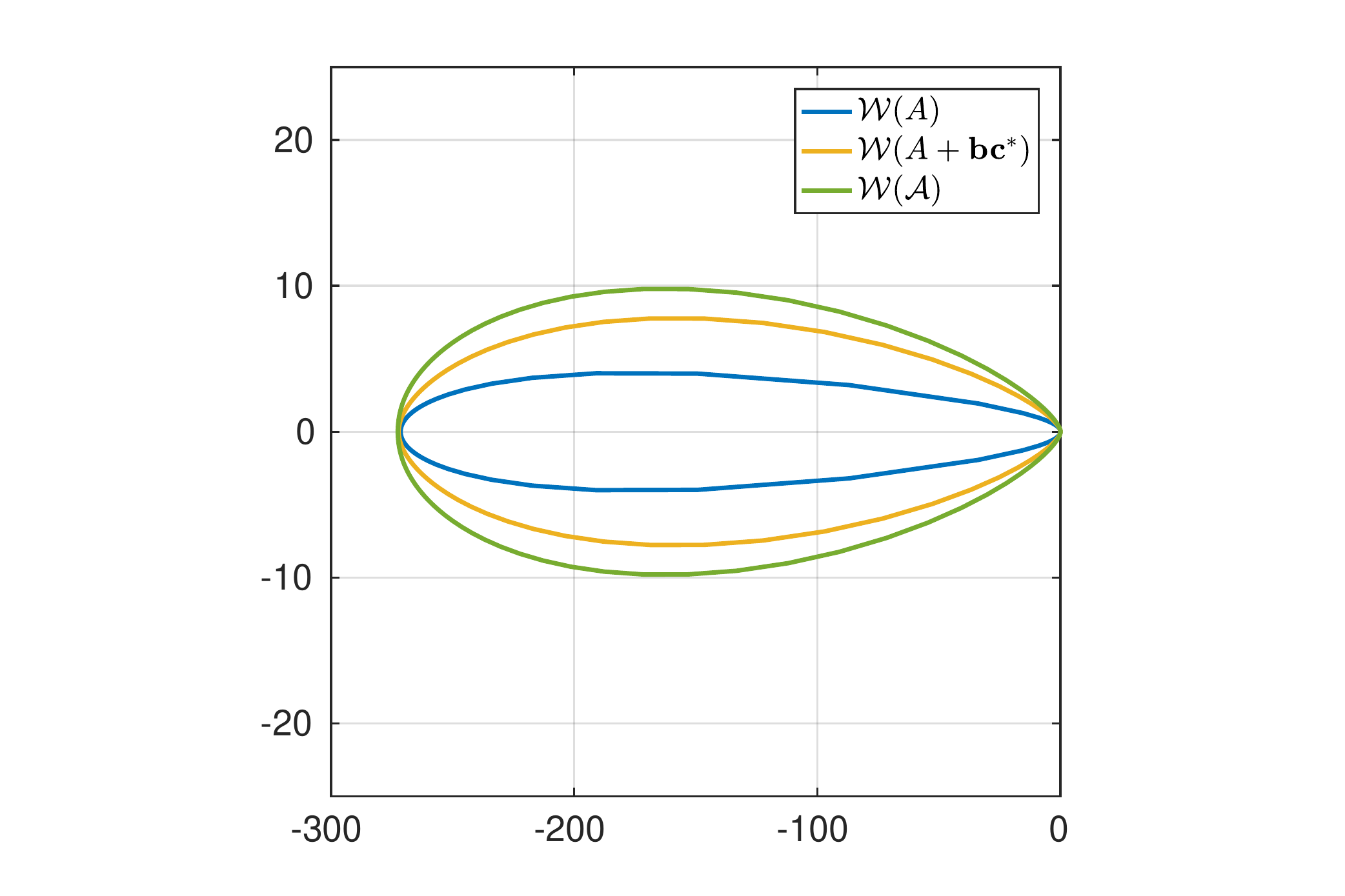}
\end{minipage}
\begin{minipage}{.49\textwidth}
\includegraphics[width=\textwidth]{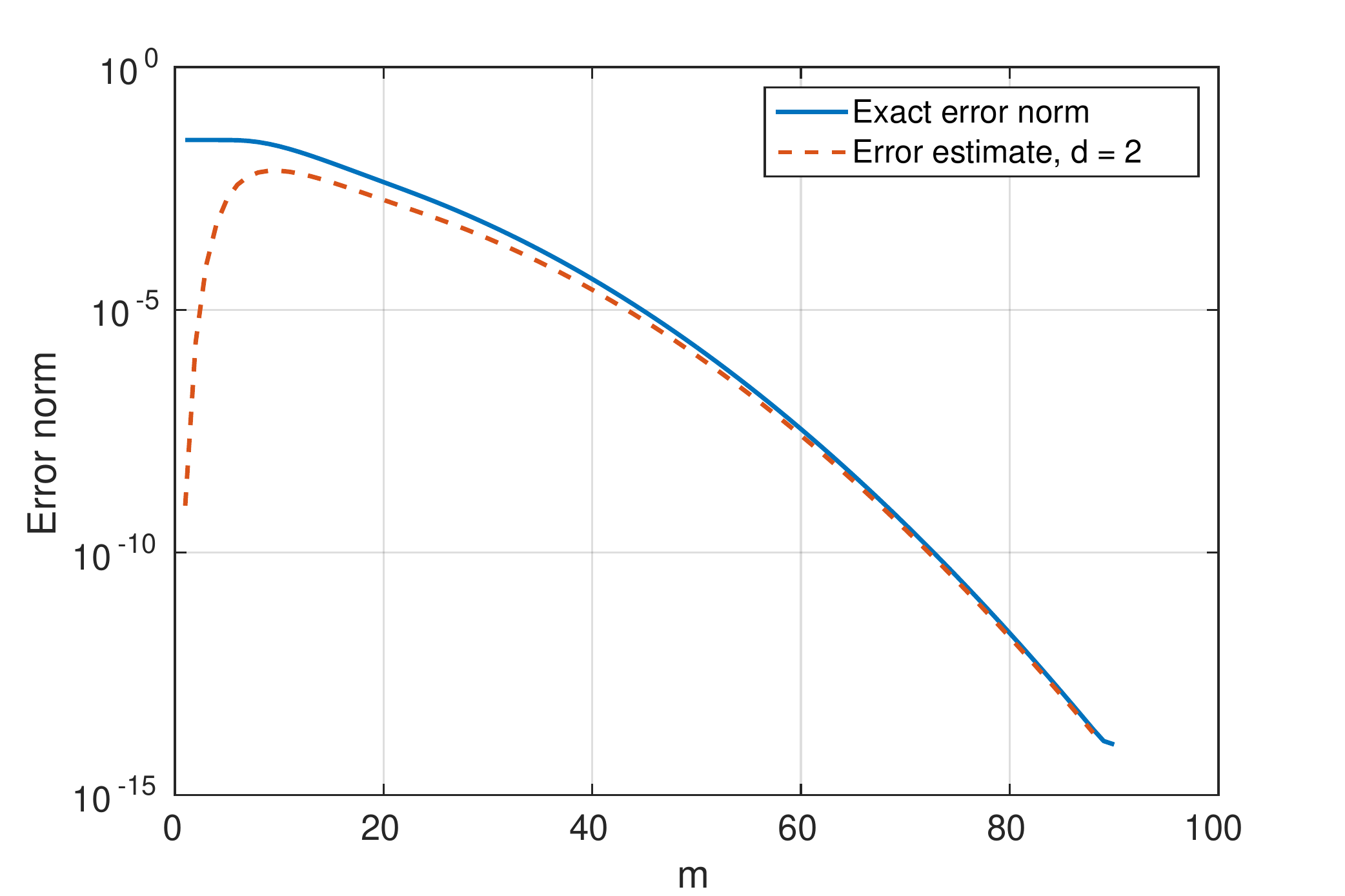}
\end{minipage}
\begin{minipage}{.49\textwidth}
\includegraphics[width=\textwidth]{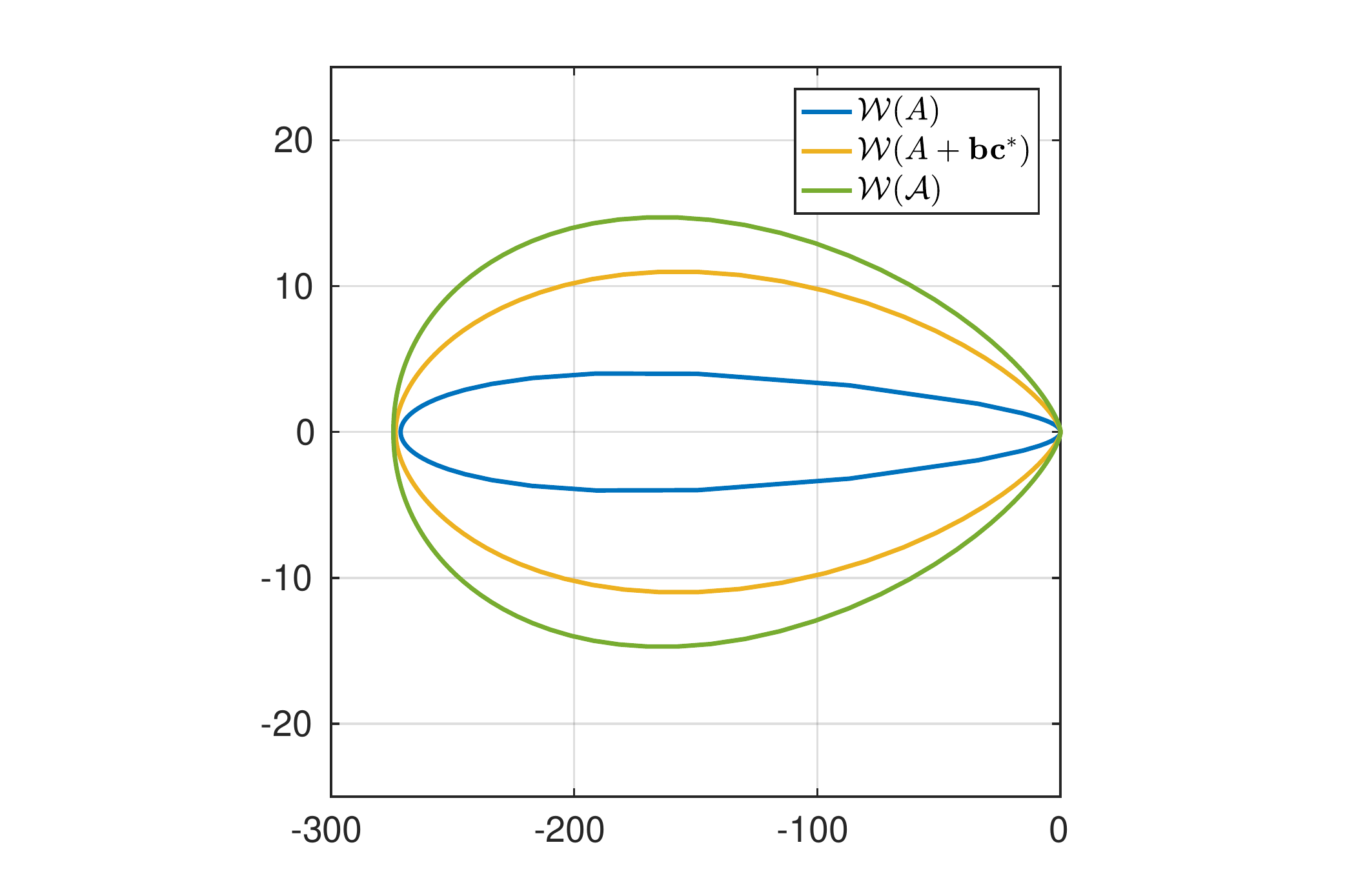}
\end{minipage}
\begin{minipage}{.49\textwidth}
\includegraphics[width=\textwidth]{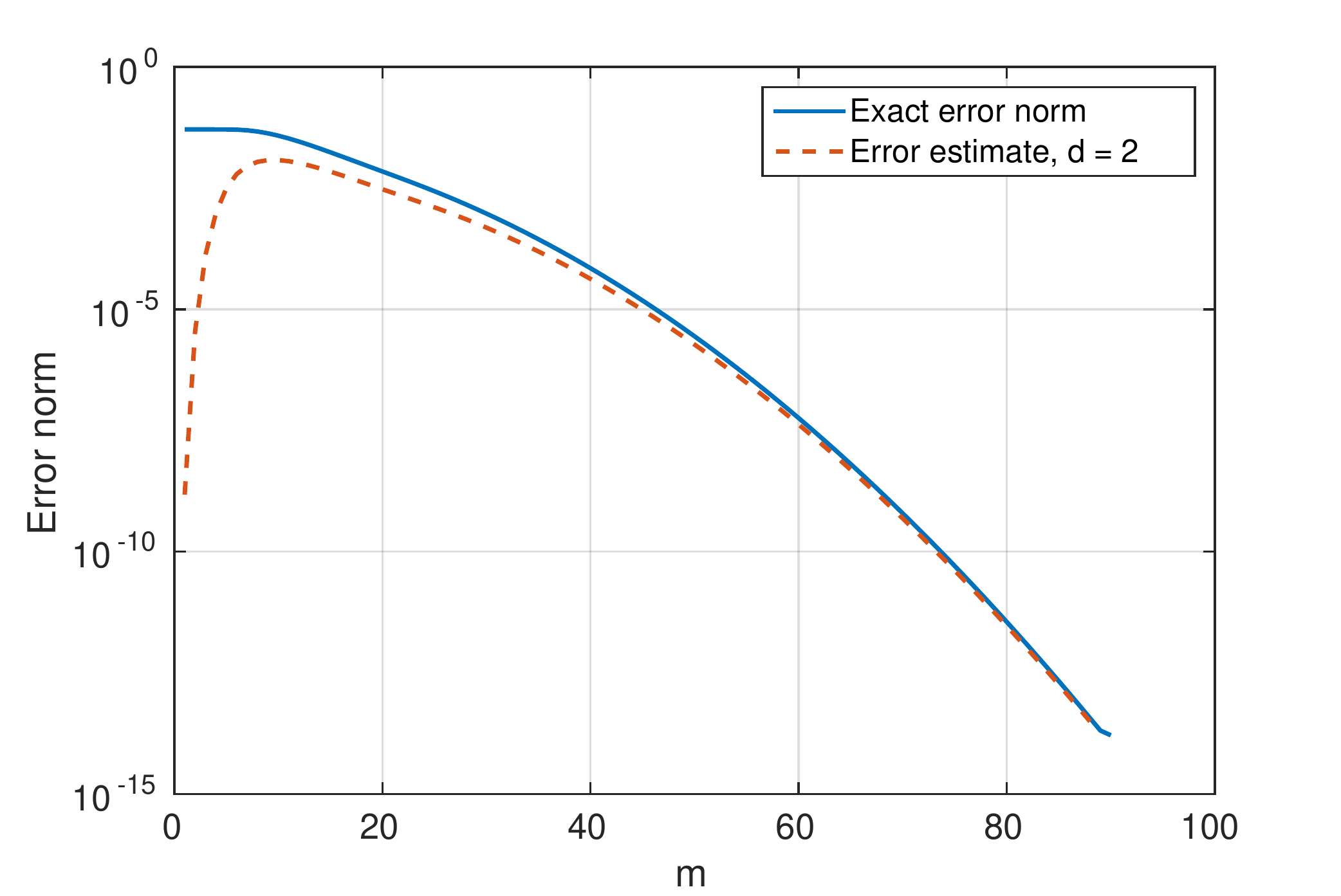}
\end{minipage}
\caption{Field of values (left) for the modifications of the discretized convection diffusion operator corresponding to $\widetilde{c} = 20,40,60$ (top to bottom) and corresponding convergence curves (right) for approximating $\exp(A+\vb\vc^\ast)-\exp(A)$. \label{fig:convdiff}}
\end{figure}

In the following, we illustrate the convergence of Algorithm~\ref{alg:krylovnonhermitian} for a non-Hermitian example.
Consider the one-dimensional convection diffusion equation
\begin{align*}
u^{\prime\prime} - c \cdot u^\prime &= f \text{ in } [0,1], \\
                             u(0) = u(1) &= 1,
\end{align*}
discretized by centered finite differences combined with a second-order scheme for the convection term. We choose the convection coefficient to be $c = 10$ and discretize the equation on 256 interior grid points. We then consider the rank-one modifications $A+\vb\vc^\ast$ of the discretized operator that change the convection coefficient $c$ in the middle, at the $128$th grid point, to (a) $\widetilde{c} = 20$, (b) $\widetilde{c} = 40$, (c) $\widetilde{c} = 60$. On the left-hand side of Figure~\ref{fig:convdiff}, the field of values of $A$, $A+\vb\vc^\ast$ and the matrix $\calA$ from~\eqref{eq:calA} is shown for these modifications, and on the right-hand side the convergence history of Algorithm~\ref{alg:krylovnonhermitian} for approximating $\exp(A+\vb\vc^\ast)-\exp(A)$ is given.

As shown in Figure~\ref{fig:convdiff}, the field of values of $A$ lies in a wedge-like set in left half-plane, as considered in Corollary~\ref{cor:wedge}, with a rather small inner angle at the origin. The angle increases when considering $A+\vb\vc^\ast$ and this increase becomes more pronounced as $\widetilde{c}$ grows. In light of Corollary~\ref{cor:wedge}, we would thus expect slower convergence for larger values of $\widetilde{c}$. However, as the convergence curves on the right-hand side of Figure~\ref{fig:convdiff} show, the convergence behavior is virtually identical for all three test cases. Therefore, our convergence estimates cannot be expected to be sharp in this case. We also report the shape of $\calW(\calA)$ in Figure~\ref{fig:convdiff} in order to illustrate the superiority of the approach from Section~\ref{sec:convergence_integral} over the result of Theorem~\ref{the:convergence_polynomial_rational_nonhermitian}, which depends on the field of values of $\calA$. Comparing $\calW(A+\vb\vc^\ast)$ and $\calW(\calA)$, we observe that the latter set is substantially larger, especially for large values of $\widetilde{c}$. An especially unfavorable case appears for $\widetilde{c} = 60$, where it turns out that $\calW(\calA)$ is \emph{not} a subset of the left complex half-plane anymore (and contains $0$), unlike $\calW(A)$ and $\calW(A+\vb\vc^\ast)$. For an entire function like the exponential function, this only leads to worse convergence bounds. For other functions, like the inverse square root having a singularity at $0$, this can be more problematic, to the extent that Theorem~\ref{the:convergence_polynomial_rational_nonhermitian} is not applicable.

Figure~\ref{fig:convdiff} also shows the difference-based error estimate~\eqref{eq:error_estimate_difference} for $d=2$. While this estimate closely follows the exact error curve in later iterations, it severely underestimates the error in the first few iterations. This is due to the fact that the method almost stagnates in these iterations, and is a typical shortcoming of difference-based error estimates.

\section{Conclusions}\label{sec:conclusion}
We have proposed Krylov subspace methods for approximating $f(A+D)-f(A)$ for a low-rank matrix $D$. We proved that the resulting approximation is exact for a polynomial of a certain degree and used this to derive a variety of convergence results, which either link the convergence of our method to polynomial approximation problems or exploit results on the error in the full orthogonalization method in conjunction with an integral representation of $f$. 

In numerical experiments--in particular on applications from network analysis---we have illustrated that our approach can dramatically reduce the cost of computing $f(A+D)$ or portions thereof.
We expect that our algorithms will prove useful in other application areas and we will explore this in future work.
Another interesting topic for future research is the use of extended and rational Krylov subspaces in the proposed method and an analysis of the resulting convergence behavior.

\bibliographystyle{siam}
\bibliography{matrixfunctions}

\end{document}